\documentclass{amsart}
\usepackage{amssymb}
\usepackage{stmaryrd}
\usepackage{enumerate}
\usepackage{cite}
\usepackage{hyperref}
\usepackage{amsthm}
\usepackage{tikz-cd}
\usepackage{colonequals}
\usepackage{IEEEtrantools}
\newcommand{\QQ}{\mathbb{Q}}
\newcommand{\Qp}{\mathbb{Q}_p}

\newcommand{\Zp}{\mathbb{Z}_p}
\newcommand{\Cp}{\mathbb{C}_p}
\newcommand{\RR}{\mathbb{R}}
\newcommand{\ZZ}{\mathbb{Z}}
\newcommand{\NN}{\mathbb{N}}
\newcommand{\db}[1]{\llbracket #1 \rrbracket}
\renewcommand{\AA}{\mathcal{A}}
\newcommand{\Af}{\mathbb{A}_f}
\newcommand{\Afp}{\mathbb{A}_f^p}
\newcommand{\DD}{\mathcal{D}}
\newcommand{\OO}{\mathcal{O}}
\newcommand{\OCp}{\OO_{\Cp}}
\newcommand{\WW}{\mathcal{W}}
\newcommand{\CC}{\mathcal{C}}
\newcommand{\LL}{\mathcal{L}}
\newcommand{\Fp}{\mathbb{F}_p}

\newcommand{\ang}[1]{\left< #1 \right>}
\newcommand{\fs}[1]{\left\{ \left\{ #1 \right\} \right\}}

\newcommand{\Aar}[1][r]{\mathcal{A}^{(\alpha,#1)}}
\newcommand{\Dar}[1][r]{\mathcal{D}^{(\alpha,#1)}}
\newcommand{\AD}{\mathbb{A}}
\newcommand{\isom}{\xrightarrow{\sim}}

\newcommand{\HH}{\mathcal{H}}
\newcommand{\ur}{\underline{r}}
\newcommand{\EE}{\mathcal{E}}
\newcommand{\UU}{\mathcal{U}}
\newcommand{\VV}{\mathcal{V}}
\newcommand{\ZC}{\mathcal{Z}}
\newcommand{\MM}{\mathcal{M}}
\newcommand{\PP}{\mathcal{P}}
\newcommand{\KK}{\mathcal{K}}
\newcommand{\XX}{\mathcal{X}}

\newcommand{\cusp}{\mathrm{cusp}}

\newcommand{\fl}[1]{\left\lfloor #1 \right\rfloor}
\newcommand{\cl}[1]{\left\lceil #1 \right\rceil}
\newcommand{\an}{\mathrm{an}}
\newcommand{\bn}{\mathbf{n}}
\newcommand{\bt}{\mathbf{t}}
\newcommand{\hu}{h_{\UU,Q_+}}
\newcommand{\huo}{h_{\UU,Q_+,0}}

\DeclareMathOperator{\der}{der}
\DeclareMathOperator{\Spa}{Spa}
\DeclareMathOperator{\Hom}{Hom}
\DeclareMathOperator{\Ind}{Ind}
\DeclareMathOperator{\tr}{tr}
\DeclareMathOperator{\End}{End}
\DeclareMathOperator{\id}{id}
\DeclareMathOperator{\Coeff}{Coeff}
\DeclareMathOperator{\Ser}{Ser}
\DeclareMathOperator{\Ev}{Ev}
\DeclareMathOperator{\Gal}{Gal}
\DeclareMathOperator{\SSym}{SSym}
\DeclareMathOperator{\Sym}{Sym}
\newcommand{\Gadq}{G(\QQ)/Z_G(\QQ)}
\newcommand{\wtoff}{\rho+w(P_{\mu})^{-1}(\rho-2\rho_M)}
\theoremstyle{plain}
\numberwithin{equation}{subsection}
\newtheorem{lem}[equation]{Lemma}
\newtheorem{prop}[equation]{Proposition}
\newtheorem{thm}[equation]{Theorem}
\newtheorem{cor}[equation]{Corollary}
\theoremstyle{definition}
\newtheorem{defn}[equation]{Definition}
\theoremstyle{remark}
\newtheorem{ex}[equation]{Example}
\newtheorem{rmk}[equation]{Remark}

\title{Equidimensional adic eigenvarieties for groups with discrete series}
\author{Daniel R. Gulotta}
\address{Department of Mathematics, Columbia University, New York, NY 10027, USA}
\curraddr{Max Plank Institute for Mathematics, Vivatsgasse 7, 53119 Bonn, Germany}
\email{gulotta@mpim-bonn.mpg.de}
\begin{document}
\begin{abstract}
We extend Urban's construction of eigenvarieties for reductive groups $G$ such that
$G(\RR)$ has discrete series to include characteristic $p$ points at the boundary of weight
space.  In order to perform this construction, we define a notion of
``locally analytic'' functions and distributions on a locally $\Qp$-analytic
manifold taking values in a complete Tate $\Zp$-algebra in which $p$ is not
necessarily invertible.  Our definition agrees with the definition of
locally analytic distributions on $p$-adic Lie groups given by Johansson
and Newton.
\end{abstract}
\maketitle
\tableofcontents
\section{Introduction}
\subsection{Statement of results}

The study of $p$-adic families of automorphic forms began with the work
of Hida \cite{Hid86,Hid88,Hid94}.  Coleman and Mazur \cite{CM96,Col96,Col97} introduced the
eigencurve, which parameterizes overconvergent $p$-adic modular forms of finite slope.
Coleman and Mazur used a geometric definition of $p$-adic modular forms, based on the
original definition of Katz \cite{Kat73}.  It is also possible to define $p$-adic automorphic
forms using a cohomological approach.
Several constructions of eigenvarieties are based on overconvergent cohomology, introduced by
Stevens \cite{Ste94} and later generalized by Ash-Stevens \cite{AS08}.
These include the constructions of Urban \cite{Urb11} and Hansen \cite{Han15}.
Emerton \cite{Eme06C} has also constructed eigenvarieties using a somewhat different
cohomological approach.

The eigenvarieties mentioned above are all rigid analytic spaces, so they parameterize
forms that have coefficients in $\Qp$-algebras.  Recently, there has been interest in
studying forms with coefficients in characteristic $p$.
Liu, Wan, and Xiao \cite{LWX17} constructed $\Zp \db{\Zp^{\times}}$-modules of
automorphic forms for definite quaternion algebras.
By taking quotients of this module, one can obtain both traditional $p$-adic automorphic forms
and forms with coefficients in $\mathbb{F}_p \db{\Zp^{\times}}$ whose
existence had been conjectured by Coleman.  Using these modules,
Liu, Wan, and Xiao proved certain cases of a conjecture of Coleman-Mazur and
Buzzard-Kilford \cite{BK05}
concerning the eigenvalues of the $U_p$ operator near the boundary of the
weight space.  Andreatta, Iovita, and Pilloni \cite{AIP15} constructed an
eigencurve that included characteristic $p$ points by extending Katz's
definition of $p$-adic modular forms.

In this paper, we will show how Urban's eigenvarieties can be extended to include the
characteristic $p$ points at the boundary of weight space.

In order to explain our results in more detail, we will first describe the
basic idea of overconvergent cohomology.
Let $G$ be a connected reductive algebraic
group over $\QQ$ such that $G_{\Qp}$ is quasisplit.
Let $\AD$ be the adeles over $\QQ$,
let $\Afp$ be the finite adeles away from $p$,
let $G_{\infty}^+$ be the identity component of $G(\RR)$,
and let $Z_G$ be the center of $G$.
Let $T_0$ be a maximal compact torus of $G(\Qp)$, and let $N_0^-$ be
an open compact subgroup of a maximal unipotent subgroup of $G(\Qp)$.
We may consider the space
\[ \mathcal{X} \colonequals G(\AD)/K^p G_{\infty}^+ \]
as a locally $\Qp$-analytic manifold.
Let $F$ be a finite extension of $\Qp$, and let $\lambda \colon T_0 \to F^{\times}$
be a continuous homomorphism.
Let $\DD_{c,\lambda}$ be the space of compactly supported
$F$-valued locally analytic distributions
on $\mathcal{X}$, modulo the relations that right translation by $N_0^-$ acts
as the identity, right translation by $T_0$ acts by $\lambda$, and
translation by $Z_G(\QQ)$ acts by the identity.
One may think of the cohomology groups $H^i(\Gadq,\DD_{c,\lambda})$
as spaces of $p$-adic automorphic forms.
One can also study families of $p$-adic automorphic forms by
replacing $F$ with an affinoid $\Qp$-algebra $A$.

We are interested in extending overconvergent cohomology to the case
where $A$ is a $\Zp$-algebra.  The main challenge is to show that
there is a suitable notion of locally analytic $A$-valued functions
and distributions on $\mathcal{X}$.
We will define these notions when $A$ is a complete
Tate $\Zp$-algebra.

To see what the definition should be, we recall a fact from $p$-adic functional analysis: a
function $f \colon \Zp \to \Qp$ is locally analytic if and only if it is of the form
$f(z) = \sum_{n=0}^{\infty} a_n \binom{z}{n}$,
where $a_n \in A$ and $|a_n|_p$ go to zero exponentially as $n \to \infty$.
We will therefore define the space $\AA(\Zp,A)$ of ``locally analytic'' functions $\Zp \to A$
to be the set of functions of the form $\sum_{n=0}^{\infty} a_n \binom{z}{n}$, where $a_n \in A$
and $a_n$ to to zero exponentially (i.~e.\ $\alpha^{-n} a_n$ goes to zero for
some topologically
nilpotent unit $\alpha$) as $n \to \infty$.
If $p$ is invertible in $A$, then this definition is known to coincide with the usual one.
We will make a similar definition for locally analytic functions $\Zp^k \to A$, and then
extend the definition to locally $\Qp$-analytic manifolds by gluing.

If $X$ is a locally $\Qp$-analytic manifold, then we will define modules
$\AA(X,A)$, $\DD(X,A)$, $\AA_c(X,A)$, $\DD_c(X,A)$ of locally analytic functions,
distributions, compactly supported functions, and compactly supported distributions, respectively.
\begin{thm}[Theorem \ref{thm:la}] \label{thm:introla}
The modules $\AA(X,A)$, $\DD(X,A)$, $\AA_c(X,A)$, and $\DD_c(X,A)$ satisfy the following
properties:
\begin{enumerate}
\item $\AA(X,A)$ is ring.
\item If $g \colon X \to Y$ is a locally analytic map, then composition with
$g$ induces homomorphisms $\AA(Y,A) \to \AA(X,A)$ and
$\DD(X,A) \to \DD(Y,A)$.
\item The functors $U \mapsto \AA(U,A)$ and $U \mapsto \DD_c(U,A)$ are sheaves on $X$.
\item If $X$ has the structure of a finitely generated $\Zp$-module, then any continuous
group homomorphism $X \to A^{\times}$ is in $\AA(X,A)$.
\end{enumerate}
\end{thm}
\begin{rmk}
Of course, modules of continuous functions and distributions also
satisfy the above properties.  What makes $\AA(X,A)$ and $\DD(X,A)$
more like modules of locally analytic functions and distributions is that a map
that multiplies all coordinates by $p$ is ``completely continuous'';
see Proposition \ref{prop:local} and Lemma \ref{lem:inccc} for the precise statement.
\end{rmk}

In \cite{Urb11}, Urban constructed eigenvarieties for reductive groups $G$ such that
$G(\RR)$ has discrete series.
We will show how to use the locally analytic distribution modules mentioned above
to extend Urban's construction to include characteristic $p$ points.
\begin{thm}[Theorem \ref{thm:eigenvariety}] \label{thm:introeig}
The reduced eigenvariety constructed in \cite{Urb11} extends to an adic space
$\EE$ over the weight space $\WW=\Spa(\Zp \db{T'},\Zp \db{T'})^{\an}$, where
$T'$ is a quotient of a compact subgroup of a maximal torus in $G(\Qp)$.
Furthermore, $\EE$ is equidimensional and the projection from $\EE$ to the spectral variety
$\ZC$ is finite and surjective.
\end{thm}

The spectral variety $\ZC$ is flat over $\WW$, so the existence of characteristic $p$ points of $\ZC$
implies the existence of nearby characteristic zero points.  It seems to be a difficult problem to prove the
existence of boundary points in general; however, in many cases, one can check explicitly that they exist
(see for example \cite{LWX17, birkbeck-slopes, jn-parallel, ye-slopes}).

As this work was being prepared, I became aware that Christian Johansson and James Newton
were independently pursuing similar work.  In \cite{JN16}, they adapt Hansen's construction of
eigenvarieties to include the boundary of weight space.  Their definition of locally analytic
distributions on $\Zp^k$ is essentially the same as ours.  To construct distributions on $p$-adic
Lie groups, they use a particular choice of coordinate charts previously studied by
Schneider and Teitelbaum.  Our definition of locally analytic distributions therefore generalizes theirs.

\subsection{Summary of Urban's construction and outline of the paper}
Urban's construction of eigenvarieties is based on the framework of
overconvergent cohomology developed by Stevens and Ash-Stevens.
In this framework, one first defines a weight space $\WW$, as mentioned above.
Over any affinoid subspace $\UU$ of the weight space $\WW$,
one defines a complex $C^{\bullet}$ of projective $\OO_{\WW}(\UU)$-modules.
The cohomology groups of this complex are the groups $H^i(\Gadq,\DD_{c,\lambda})$
mentioned above, where $\lambda$ is the composition of the quotient
$T_0 \to T'$ with the tautological character
$T' \to \OO_{\WW}(\UU)^{\times}$.

We consider a weight space that is larger than the one considered by
Ash-Stevens and Urban.  In particular, our weight space contains opens
$\UU$ such that the prime $p$ is not invertible on $\OO_{\WW}(\UU)$.
The main challenge in defining overconvergent cohomology over the
larger weight space is to find a suitable notion of ``locally analytic''
$\OO_{\WW}(\UU)$-valued distributions.
After recalling the necessary background in section \ref{sec:mod},
we define modules of locally analytic functions and distributions
and prove some properties of these modules in section \ref{sec:la}.
We use these modules to define overconvergent cohomology in section \ref{sec:automorphic}.

Ash-Stevens proposed constructing an eigenvariety whose points correspond to
systems of Hecke eigenvalues appearing in the cohomology of the complexes
$C^{\bullet}$;
Hansen's construction uses this approach.
Urban took a more $K$-theoretic approach.
Assume that $G(\RR)$ has discrete series; then
cuspidal automorphic forms of regular weight contribute to a single degree $q_0$ of the cohomology
of $C^{\bullet}$.
So the associated systems of Hecke eigenvalues appear a net positive number of times in the
formal sum $\sum_i (-1)^{i-q_0} C^i$.
Urban showed that after removing the contributions from Eisenstein series, each
system of Hecke eigenvalues appears a net nonnegative number of times in the
formal sum.
The points in Urban's eigenvariety correspond to those systems of
eigenvalues appearing a net positive number of times in the formal sum.

Unfortunately, Urban's analysis of Eisenstein series contained an error.
In order to argue that certain character distributions are uniquely defined,
Urban assumed that the region of convergence of an Eisenstein series is (up
to translation) a union of Weyl chambers.  However, this assumption is not true.
In section \ref{sec:eisenstein}, we correct this error by giving a new argument for uniqueness.

Urban's construction of eigenvarieties makes use of the theory of pseudocharacters.
We will instead use Chenevier's theory of determinants \cite{Che14},
which is equivalent to the theory of pseudocharacters in characteristic zero but
better behaved in our setting where the prime $p$ may not be invertible.
Section \ref{sec:determinant} recalls some basic facts about determinants and proves
some criteria for establishing that a ratio of two determinants is again a determinant.

Finally, in section \ref{sec:eigenvariety}, we construct the eigenvariety.
We adapt Urban's construction from the setting of rigid analytic
spaces to the setting of adic spaces.

\section*{Acknowledgments}

I would like to thank my advisor, Eric Urban, for suggesting that I consider
the problem of adapting cohomological constructions of eigenvarieties to
include the boundary of weight space, and for helpful discussions.  I would
also like to thank David Hansen and Shrenik Shah for helpful discussions.
I would like to thank the anonymous referee for pointing out some mistakes
and providing many helpful suggestions for improving the exposition.
Discussions with Ana Caraiani and Christian Johansson prompted me to make some
improvements to the exposition in section \ref{sec:eisenstein}.
I would like to thank Sam Mundy for pointing out several mistakes in
section \ref{sec:eisenstein}.
I would like to thank Christian Johansson and James Newton for informing me about their work.
Some corrections to this paper were made during my time at Oxford (supported by Royal Society grant RP\textbackslash{}EA\textbackslash{}180020) and at MPIM.

\section{Modules over complete Tate rings} \label{sec:mod}

\subsection{Definitions}

We begin by recalling the framework necessary for defining
modules of locally analytic functions and distributions and for
defining eigenvarieties.
We will repeat the basic setup of \cite[\S 2]{Buz07} \cite[Appendice B]{AIP15}.
First, we recall the definition of a Tate ring \cite[\S 1]{Hub93}.
\begin{defn}
A \emph{Huber ring} is a topological ring $A$ such that there exists an open subring
$A_0 \subset A$ and a finitely generated ideal $I \subset A_0$ such that $A_0$ has
the $I$-adic topology.  We say that $A_0$ is a \emph{ring of definition} of $A$
and $I$ is an \emph{ideal of definition} of $A_0$.
 
A \emph{Tate ring} is a Huber ring $A$ such that some (equivalently, any) ring of definition
$A_0$ has an ideal of definition that is generated by a topologically nilpotent unit of $A$.
\end{defn}
In section \ref{sec:eigenvariety}, we will use the framework of adic
spaces to construct the eigenvariety.  Every analytic adic space
can be covered by open subsets of the form $\Spa(A,A^+)$ with $A$
complete Tate, so it is natural to consider this class of rings.

Throughout this section, $A$ will denote a complete Tate ring.

\begin{defn}
Let $X$ be a quasi-compact topological space, and let $M$ be a topological
abelian group.  We define $\CC(X,M)$ to be the space of continuous functions
$X \to M$, with the topology of uniform convergence.
\end{defn}

\begin{defn}
Let $S$ be a set, and let $M$ be a topological abelian group.
We define $c(S,M)$ to be the space of functions $f \colon S \to M$ such that
for any open neighborhood $U$ of the identity in $M$, the complement
of $f^{-1}(U)$ is finite.
We give $c(S,M)$ the topology of uniform convergence.
\end{defn}

\begin{defn}
Let $M$ be a topological $A$-module.  We say that $M$ is
\emph{orthonormalizable} if it is isomorphic to
$c(S,A)$ for some set $S$.  We say that $M$ is \emph{projective} if
it is a direct summand of an orthonormalizable $A$-module.
\end{defn}

\begin{defn}
Let $M$ be a topological $A$-module.  We say that a set
$B \subset M$ is \emph{bounded} if for all open neighborhoods $U$ of the
identity in
$M$, there exists $\alpha \in A^{\times}$ so that $\alpha B \subseteq U$.
\end{defn}

\begin{defn}
Let $M$ and $N$ be topological $A$-modules.
We define $\LL_b(M,N)$ to be
the set of continuous $A$-module homomorphisms $M \to N$, with the topology
of convergence on bounded subsets.
\end{defn}

\begin{defn}
Let $M$ and $N$ be topological $A$-modules.
We say that an $A$-module homomorphism $M \to N$ has
\emph{finite rank} if its image is a finitely presented $A$-module.
We say that an element of $\LL_b(M,N)$ is
\emph{completely continuous} if it is in the closure of the subspace of
finite rank elements.
\end{defn}

\subsection{Spectral theory} \label{sec:spectral}

\begin{defn} \label{def:fredholmtate}
We define $A \fs{X}$ to be the set of power series
$P(X)=\sum_{n=0}^{\infty} a_n X^n$, $a_n \in A$,
such that for any $\alpha \in A^{\times}$, $\alpha^{-n} a_n \to 0$ as $n \to \infty$.

We say that $P(X) \in A \fs{X}$ is a \emph{Fredholm series} if it has
leading coefficient $1$.
\end{defn}
In section \ref{sec:eigenvariety}, we will consider the adic space
$\Spa(A,A^+)$ for certain complete Tate $\Zp$-algebras $A$.
If the adic space $\Spa(A,A^+) \times \mathbb{A}^1$ exists, then
$A \fs{X}$ is its ring of global sections.

Assume $A$ is Noetherian.
If $M$ is a projective $A$-module and $u \colon M \to M$ is completely
continuous, then we define the Fredholm series
$\det(1-Xu) \in A \fs{X}$ as in \cite[\S 2]{Buz07} \cite[\S B.2.4]{AIP15}.  
To define the series, we express $M$ as a direct summand
of an orthonormalizable $A$-module
$c(S,A)$ and extend $u$ to a map $c(S,A) \to c(S,A)$ by having it act as zero
on the orthogonal complement of $M$.
The module $c(S,A)$ has a basis consisting of functions sending
a single element of $S$ to $1$ and the rest to $0$.
We consider the matrix of $u$ in this basis.
The series $\det(1-Xu)$ is defined to be limit of the characteristic
polynomials of finite dimensional submatrices of this matrix.
The series does not depend on the choice of embedding.

As in \cite{Urb11}, we will need to work with complexes.
Let $M^{\bullet}$ be a bounded
complex of projective $A$-modules.  We will
say that $u^{\bullet} \colon M^{\bullet} \to M^{\bullet}$ is
completely continuous if each $u^i$ is completely
continuous.  If $u^{\bullet}$ is completely continuous,
then we define
\[ \det(1-Xu^{\bullet}) \colonequals \prod_i \det(1-Xu^i)^{(-1)^i} \,. \]

\begin{lem} \label{lem:detinvt}
Let $M^{\bullet}$ be a bounded complex of projective $A$-modules,
and let $u^{\bullet}, v^{\bullet} \colon M^{\bullet} \to M^{\bullet}$
be completely continuous maps that are homotopy equivalent.
Then $\det(1-Xu^{\bullet}) = \det(1-Xv^{\bullet})$.
\end{lem}
\begin{proof}
For each nonnegative integer $k$,
define the complex $\SSym^k M^{\bullet}$ so that
$(\SSym^k M)^i$ is generated by formal products of $k$ homogeneous elements of $M^{\bullet}$
of total degree $i$, subject to the relation a pair of homogeneous elements
anticommutes if both have odd degree and commutes otherwise.
The differential $d \colon (\SSym^k M)^i \to (\SSym^k M)^{i+1}$ is defined by
$d(m_1 m_2 \dotsm m_k) = (dm_1) m_2 \dotsm m_k + (-1)^{\deg m_1} m_1 (dm_2) \dotsm m_k + \dotsb + (-1)^{i - \deg m_k} m_1 m_2 \dotsm (dm_k)$.
The maps $u^{\bullet}$, $v^{\bullet}$ induce endomorphisms
$\SSym^k u^{\bullet}$, $\SSym^k v^{\bullet}$ on
$\SSym^k M^{\bullet}$, and these are completely continuous and
homotopy equivalent.
We claim that the coefficient of $X^k$ in $\det(1-Xu^{\bullet})^{-1}$
is $\tr \SSym^k u^{\bullet}$.
Indeed, there is a decomposition
\begin{align*}
\sum_{k=0}^{\infty} X^k \tr \SSym^k u^{\bullet} & = \prod_{i \equiv 0(2)} \left( \sum_{k=0}^{\infty} X^k \tr \Sym^k u^i \right) \prod_{i \equiv 1 (2)} \left( \sum_{k=0}^{\infty} (-X)^k \tr \wedge^k u^i \right) \\
& = \prod_{i \equiv 0(2)} \det(1-X u^i)^{-1} \prod_{i \equiv 1 (2)} \det(1-Xu^i) \,.
\end{align*}

Therefore it suffices to show that for each $k$, $\SSym^k u^{\bullet}$ and
$\SSym^k v^{\bullet}$ have the same trace.  Then we may use the argument
of \cite[Lemma 2.2.8]{Urb11}.
\end{proof}

\subsection{Norms} \label{sec:norm}

It is often convenient to work with norms on $A$ and on $A$-modules.
\begin{defn}
Let $\alpha$ be a topologically nilpotent unit of $A$.
We define an \emph{$\alpha$-Banach norm} on $A$
to be a continuous map $|\cdot| \colon A \to \RR^{\ge 0}$ satisfying
the following conditions.
\begin{itemize}
\item $|a+b| \le \max(|a|,|b|) \quad \forall a,b \in A$
\item $|ab| \le |a||b| \quad \forall a,b \in A$
\item $|0|=0,\, |1|=1,\, |\alpha| |\alpha^{-1}| = 1$
\item The norm $|\cdot|$ induces the topology of $A$.
\end{itemize}
\end{defn}

\begin{defn}
Let $\alpha$ be a topologically nilpotent unit of $A$, let
$|\cdot|$ be an $\alpha$-Banach norm on $A$, and let
$M$ be a topological $A$-module.  We define a \emph{$|\cdot|$-compatible norm}
on $M$ to be a continuous map $\| \cdot \| \colon M \to \RR^{\ge 0}$ satisfying
the following conditions.
\begin{itemize}
\item $\|m+n\| \le \max(\|m\|,\|n\|) \quad \forall m,n \in M$
\item $\|am\| \le |a|\|m\| \quad \forall a \in A, \, m \in M$
\item $\|0\| = 0$
\end{itemize}
If, in addition, $\| \cdot \|$ induces the topology of $M$, we say that
$\| \cdot \|$ is a \emph{Banach norm}.
\end{defn}

For any topologically nilpotent unit $\alpha \in A^{\times}$ and ring
of definition $A_0$ of $A$ containing $\alpha$,
the function $|\cdot| \colon A \to \RR^{\ge 0}$ defined by
\[ |a|=\inf_{n \in \ZZ | \alpha^n a \in A_0} p^{n} \]
is an $\alpha$-Banach norm.

Furthermore, if $M$ is a topological $A$-module and $M_0$ is an open neighborhood
of zero in $M$ that is an $A_0$-module, then the function $\|\cdot\| \colon M \to \RR^{\ge 0}$
defined by 
\[ \|m\|=\inf_{n \in \ZZ | \alpha^n m \in M_0} p^{n} \]
is a norm compatible with $| \cdot |$.  If the sets of the form $\alpha^n M_0$ are a basis
of open neighborhoods of zero, then this norm is Banach.

\section{Locally analytic functions and distributions} \label{sec:la}

Now let $A$ be a complete Tate $\Zp$-algebra, and let $X$ be
a locally $\Qp$-analytic manifold.  In this section, we will
define modules $\AA(X,A)$ and $\DD(X,A)$ of ``locally analytic''
$A$-valued functions and distributions on $X$.

The space $X$ can be covered by coordinate patches isomorphic
to $\Zp^k$ for some $k$.  We will first define locally analytic
functions on these patches and then show that the construction
can be glued.

Naively, one might try to define a function $\Zp^k \to A$ to be locally
analytic if it has a power series expansion in a neighborhood of any point.
However, this definition turns out not to be suitable for applications to
overconvergent cohomology.  In section \ref{sec:hecketop}, it will be important
that any continuous homomorphism $\Zp^k \to A^{\times}$ is in $\AA(\Zp^k,A)$.
The homomorphism $\Zp \to \Fp((T))^{\times}$ that sends $z \mapsto (1+T)^z$ does
not have a power series expansion on any open subset of $\Zp$.  Our criterion
for local analyticity will instead be based on Mahler expansions.

\subsection{Preliminaries}

We will recall some basic facts from $p$-adic functional analysis.

We will make use of the completed group ring $\Zp \db{\Zp^k} = \varprojlim_n \Zp [\Zp^k/p^n \Zp^k]$.

For $z \in \Zp^k$, let $[z]$ denote the corresponding group-like element of $\Zp \db{\Zp^k}$,
and let $\Delta_{z} = [z]-[0]$.  Let $I_{\Delta}$ denote the augmentation
ideal of $\Zp \db{\Zp^k}$; this is the ideal generated by the $\Delta_z$.
The ring $\Zp \db{\Zp^k}$ is local with maximal ideal $(p)+I_{\Delta}$.

We let $\Zp^k$ act on $\CC(\Zp^k,M)$ by translation: for $g \in \CC(\Zp^k,M)$,
$(zg)(y)=g(y+z)$.  This action extends to an action of $\Zp \db{\Zp^k}$.

We adopt the convention that $\NN$ is the set of nonnegative integers.
To simplify notation, if $z=(z_1,\dotsc,z_k) \in \Zp^k$, and $n=(n_1,\dotsc,n_k) \in \NN^k$, we will write
$\binom{z}{n}$ for $\prod_{i=1}^k \binom{z_i}{n_i}$,
and we will write $\sum n$ for $\sum_{i=1}^k n_i$.

\begin{lem}[Mahler's theorem, {\cite[Th\'eor\`eme II.1.2.4]{Laz65}}] \label{lem:mahler}
Let $M$ be a complete topological $\Zp$-module.  Suppose that $M$ has a basis
of open neighborhoods of zero that are subgroups of $M$.
There is an isomorphism $c(\NN^k,M) \isom \CC(\Zp^k,M)$
that sends $f \in c(\NN^k,M)$ to a function $g \in \CC(\Zp^k,M)$ defined by
\[ g(z) = \sum_{n \in \NN^k} f(n) \binom{z}{n} \,. \]
\end{lem}
We say that the right-hand side of the above equation is
the Mahler expansion of $g$.

\begin{lem}[Amice's theorem] \label{lem:la}
Let $F$ be a closed subfield of $\Cp$, and let
$LA_h(\Zp^k,F)$ be the space of functions
$\Zp^k \to F$ that extend to an analytic function $\Zp^k + p^h \OCp^k \to \Cp$.
For $f \in LA_h(\Zp^k,F)$, define
\[ |f| \colonequals \sup_{z \in \Zp^k + p^h \OCp^k} |f(z)|_p \,. \]
Then the functions
$\left\lfloor \frac{n_1}{p^h} \right\rfloor! \dotsm \left\lfloor \frac{n_k}{p^h} \right\rfloor! \binom{z}{n}$
form an orthonormal basis for the Banach space $LA_h(\Zp^k,F)$.
In other words,
every $f \in LA_h(\Zp^k,F)$ can be expressed uniquely in the form
\[ f(z) = \sum_{n \in \NN^k} a_{n} \left\lfloor \frac{n_1}{p^h} \right\rfloor! \dotsm \left\lfloor \frac{n_k}{p^h} \right\rfloor! \binom{z}{n} \,, \]
and $|f|=\sup_{n \in \NN^k} |a_n|_p$.
\end{lem}
\begin{proof}
This follows from \cite[Chapitre 3]{Ami64} (see also \cite[Th\'eor\`eme I.4.7]{Col10}).
\end{proof}

The following formulas concerning the
$p$-adic valuations of $n!$, where $n$ is a nonnegative integer,
are well-known.
\[ v_p(n!) = \sum_{k=1}^{\infty} \left\lfloor \frac{n}{p^k} \right\rfloor \]
\[ \frac{n}{p-1} - \log_p (n+1) \le v_p(n!) \le \frac{n}{p-1} \]
Consequently, if $F$ is a closed subfield of $\Cp$, and $f \colon \Zp^k \to F$
is a continuous function with the Mahler expansion
$f(z)=\sum_{n \in \NN^k} a_{n} \binom{z}{n}$,
then $f$ is locally analytic if and only if
$|a_{n}|_p$ go to zero exponentially in $\sum n$.

\subsection{Definitions}

The above facts suggest that we should define a function $\Zp^k \to A$ to be
``locally analytic'' if the coefficients of its Mahler expansion
decrease to zero exponentially.

We choose a topologically nilpotent $\alpha \in A^{\times}$.

\begin{defn} \label{def:lafunc}
Let $r \in \RR^+$.  We define $\Aar(\Zp^k,A)$ to be the space of functions
$f \in \CC(\Zp^k,A)$ such that for any open neighborhood $U$ of zero in
$A$,
there exists $N \in \NN$ so that for all integers $n>N$
and all $\delta \in I_{\Delta}^n$,
$\alpha^{\lfloor -rn \rfloor} \delta f \in \CC(\Zp^k,U)$.

For any open neighborhood $U$ of zero in $A$, we define
$U_r \subset \Aar(\Zp^k,A)$ to be the set of all $f \in \Aar(\Zp^k,A)$ such that
$\alpha^{\lfloor -rn \rfloor} \delta f \in \CC(\Zp^k,U)$ for all $n \in \NN$
and all $\delta \in I_{\Delta}^n$.  We define a topology on $\Aar(\Zp^k,A)$
by making sets of the form $U_r$ a basis of open neighborhoods of zero.

We define $\AA(\Zp^k,A) \colonequals \varinjlim_r \Aar(\Zp^k,A)$.
\end{defn}
We will not choose a topology on $\AA(\Zp^k,A)$.

The connection between this definition and Mahler expansions will
be explained by Lemma \ref{lem:laconcrete}.

The definition of $\Aar(\Zp^k,A)$ is invariant under affine changes of coordinates.

For any topologically nilpotent unit $\alpha' \in A^{\times}$ and sufficiently
small $r' \in \RR^+$, $\AA^{(\alpha',r')}(\Zp^k,A)$ injects into $\Aar(\Zp^k,A)$.
So the directed systems $(\AA^{(\alpha,r)}(\Zp^k,A))_{r \in \RR^+}$ and
$(\AA^{(\alpha',r)}(\Zp^k,A))_{r \in \RR^+}$ are cofinal, and $\AA(\Zp^k,A)$ does not depend on the choice of $\alpha$.
If $F$ is a closed subfield of $\Cp$, then by Lemma \ref{lem:la},
there are continuous injections with dense image
\[\AA^{(p,1/(p-1)p^h)}(\Zp^k,F) \hookrightarrow LA_h(\Zp^k,F) \hookrightarrow \AA^{(p,r)}(\Zp^k,F) \]
for any $r<\frac{1}{(p-1)p^h}$, so the directed systems
$(\AA^{(p,r)}(\Zp^k,F))_{r \in \RR^+}$ and $(LA_h(\Zp^k,F))_{h \in \NN}$
are also cofinal.

The module $\Aar(\Zp^k,A)$ can also be defined (albeit less symmetrically)
using $\alpha$-Banach norms.  Choose a ring of definition $A_0$ of $A$ containing $\alpha$,
and define an $\alpha$-Banach norm $|\cdot| \colon A \to \RR^{\ge 0}$ as in section \ref{sec:norm}.
Define $\| \cdot \|_0 \colon \CC(\Zp^k,A) \to \RR^{\ge 0}$ by
\[ \|f\|_0 = \sup_{z \in \Zp^k} |f(z)| \,. \]
The sets $\{ f \in A | \|f\|_0 \le s \}$, $s \in \RR^{\ge 0}$, form a basis of open neighborhoods
of the identity in $A$.  Hence in Definition \ref{def:lafunc}, we can restrict our attention to
neighborhoods of this form.  Therefore
\[ \Aar(\Zp^k,A) = \left\{ f \in \CC(\Zp^k,A) \middle| \limsup_{n \to \infty} \sup_{\delta \in I_{\Delta}^n} \|\alpha^{\lfloor -rn \rfloor} \delta f\|_0 = 0 \right\} \,, \]
and the topology on $\Aar(\Zp^k,A)$ is induced by the norm $\|\cdot\|_r \colon \Aar(\Zp^k,A) \to \RR^{\ge 0}$
defined by
\[ \|f\|_r = \sup_{n \in \NN} \sup_{\delta \in I_{\Delta}^n} \|\alpha^{\lfloor -rn \rfloor} \delta f\|_0 \,. \]
The functions $\| \cdot \|_0$ and $\| \cdot \|_r$ are Banach norms compatible with $|\cdot|$.

Presumably, it would be reasonable to define $\Aar(\Zp^k,M)$ and $\AA(\Zp^k,M)$ for
any topological $A$-module $M$ that is locally convex
in the sense that for
some (equivalently, any) ring of definition $A_0$ of $A$, $M$ has a basis
of open neighborhoods of the identity that are $A_0$-modules.
(We would just replace $A$ with $M$ in the above definition.)
However,
we will not need this additional generality.

\begin{defn} \label{def:ladist}
Let $r \in \RR^+$.
We define $\Dar(\Zp^k,A)$ to be the closure of the image of
$\LL_b(\CC(\Zp^k,A),A)$ in $\LL_b(\Aar(\Zp^k,A),A)$.

We define $\DD(\Zp^k,A) = \varprojlim_r \Dar(\Zp^k,A)$.
\end{defn}
The definition of $\DD(\Zp^k,A)$ does not depend on the choice of $\alpha$.

We chose the definitions of $\Aar(\Zp^k,A)$ and $\Dar(\Zp^k,A)$ so that these modules
would be orthonormalizable, as we will now show.
\begin{lem} \label{lem:laconcrete}
There is an isomorphism  $\Ser \colon c(\NN^k,A) \isom \Aar(\Zp^k,A)$ that sends
$f \in c(\NN^k,A)$ to a function $g \in \Aar(\Zp^k,A)$ defined by
\begin{equation} \label{eq:orthoa}
g(z)=\sum_{n \in \NN^k} \alpha^{\cl{r \sum n}}
f(n) \binom{z}{n} \,.
\end{equation}
Moreover, if $c(\NN^k,A)$ is given the supremum norm and $\Aar(\Zp^k,A)$
is given the norm $\| \cdot \|_r$, then $\Ser$ is an isometry.

There is an isomorphism $\Ev \colon \Dar(\Zp^k,A) \isom c(\NN^k,A)$ that sends
$\phi \in \Dar(\Zp^k,A)$ to
a function $h \in c(\NN^k,A)$ defined by
\begin{equation} \label{eq:orthod}
h(n) = \alpha^{\cl{r \sum n}}\phi\left( \binom{z}{n} \right) \,.
\end{equation}

Hence $\Aar(\Zp^k,A)$ and $\Dar(\Zp^k,A)$ are orthonormalizable.
\end{lem}
\begin{proof}
Let $f \in c(\NN^k,A)$, and let $g$ be defined by \eqref{eq:orthoa}.  By
Mahler's theorem, $g \in \CC(\Zp^k,A)$.
We observe that for any $h \in \CC(\Zp^k,A)$ and $\delta \in \Zp \db{\Zp^k}$,
$\|\delta h\|_0 \le \|h\|_0$.
Furthermore, if $\delta \in I_{\Delta}^m$, then $\delta \binom{z}{n} = 0$ whenever $\sum n < m$.
So
\[ \begin{split}
\|\alpha^{\fl{-rm}} \delta g\|_0 \le & \sup_{\sum n \ge m} \left|\alpha^{\fl{-rm}+\cl{r \sum n}} f(n)\right| \\
\le & \sup_{\sum n \ge m} \left| f(n)\right| \,.
\end{split} \]
It follows that $g \in \Aar(\Zp^k,A)$, and $\Ser$
is has operator norm $\le 1$.

By Mahler's theorem, we can recover $f$ from $g$:
\begin{equation} \label{eq:mahler}
f(n) = \alpha^{\fl{-r \sum n}}(\Delta_{e_1}^{n_1} \dotsm \Delta_{e_k}^{n_k} g)(0) \,,
\end{equation}
where $e_1,\dotsc,e_k$ are the standard basis for $\Zp^k$.
Since
\[ |f(n)| \le \sup_{\delta \in (I_{\Delta})^{\sum n}} \| \alpha^{\fl{-r\sum n}}\delta g\|_0 \,, \]
the relation \eqref{eq:mahler} determines a map
$\Coeff \colon \Aar(\Zp^k,A) \to c(\Zp^k,A)$
that is a left inverse of $\Ser$, and $\Coeff$ has operator norm
$\le 1$.  To see that
$\Coeff$ is also a right inverse of $\Ser$, observe that
$(\Ser \circ \Coeff)(g)$ and $g$ agree on $\NN^k$, which is dense
in $\Zp^{k}$.  Since $\Ser$ and $\Coeff$ both have operator norm $\le 1$,
they must be isometries.

The map $\Ser$ induces an isomorphism
$\Ser^* \colon \LL_b(\Aar(\Zp^k,A),A) \isom \LL_b(c(\NN^k,A),A)$.
The pairing $c(\NN^k,A) \times c(\NN^k,A) \to A$ defined by
$(f,h) \mapsto \sum_{n \in \NN^k} f(n) h(n)$ identifies $c(\NN^k,A)$
isometrically with a closed submodule of $\LL_b(c(\NN^k,A),A)$.
For any $\phi \in \LL_b(\CC(\Zp^k,A),A)$,
the function $n \mapsto \phi\left(\binom{z}{n} \right)$ is bounded, so in particular
$\alpha^{\cl{r \sum n}} \phi\left(\binom{z}{n} \right) \to 0$ as
$\sum n \to \infty$.
Hence the image of $\Ser^*$ is contained in $c(\NN^k,A)$.  Furthermore,
the image contains all elements of $c(\NN^k,A)$ that are supported
on a finite subset of $\NN^k$, and these elements are dense in $c(\NN^k,A)$.
\end{proof}

Lemma \ref{lem:laconcrete} makes it clear that for $r'<r$,
there are natural injections
\[ \Dar[r'](\Zp^k,A) \hookrightarrow \LL_b(\Aar[r'](\Zp^k,A),A) \hookrightarrow \Dar(\Zp^k,A) \]
\[ \Aar(\Zp^k,A) \hookrightarrow \LL_b(\Dar(\Zp^k,A),A) \hookrightarrow \Aar[r'](\Zp^k,A) \,. \]

\subsection{Properties of locally analytic functions and distributions} \label{sec:laproperties}

In this section, we check that $\Aar(\Zp^k,A)$ has some properties that one would expect
of locally analytic functions.

\begin{lem} \label{lem:mult}
Multiplication induces a continuous map
$\Aar(\Zp^k,A) \times \Aar(\Zp^k,A) \to \Aar(\Zp^k,A)$.
\end{lem}
\begin{proof}
This follows Lemma \ref{lem:laconcrete} and the fact that for $m,n \in \NN$, $\binom{z}{n} \binom{z}{m}$ is of the form $\sum_{i=0}^{m+n} a_i \binom{z}{i}$ with $a_i \in \ZZ$.
\end{proof}

\begin{lem} \label{lem:universal uniqueness}
Let
\[ f \colon \CC(\Zp^k,\Zp) \to \CC(\Zp^j,\Zp) \]
be a $\Zp$-module homomorphism.
For any $r,s \in \mathbb{R}^+$,
there is at most one continuous $A$-linear homomorphism $\tilde{f}$ making the diagram
\begin{center}
\begin{tikzcd}
\CC(\Zp^k,\Zp) \arrow[d,"f"] \arrow[r] & \CC(\Zp^k,A) \arrow[d] & \Aar(\Zp^k,A) \arrow[d,dashed,"\tilde{f}"] \arrow[l,hook'] \\
\CC(\Zp^j,\Zp) \arrow[r] & \CC(\Zp^j,A) & \Aar[s](\Zp^j,A) \arrow[l,hook']
\end{tikzcd}
\end{center}
commute, and there is at most one continuous $A$-linear homomorphism $\tilde{f}^*$ making the diagram
\begin{center}
\begin{tikzcd}
\Hom_{\Zp}(\CC(\Zp^j,\Zp),\Zp) \arrow[d,"f^*"] \arrow[r] & \LL_b(\CC(\Zp^j,A),A) \arrow[d] \arrow[r,hook] & \Dar[s](\Zp^j,A) \arrow[d,dashed,"\tilde{f}^*"] \\
\Hom_{\Zp}(\CC(\Zp^k,\Zp),\Zp) \arrow[r] & \LL_b(\CC(\Zp^k,A),A) \arrow[r,hook] & \Dar(\Zp^k,A)
\end{tikzcd}
\end{center}
commute.
\end{lem}
If either homomorphism exists, we say that it is induced by $f$.
\begin{proof}
If the maps $\tilde{f}$ and $\tilde{f}^*$ exist, then their matrices
in the basis of Lemma \ref{lem:laconcrete} can be deduced from the matrix
of $f$ in the basis of Mahler's theorem.  More specifically,
if we write
\[ f\left( \binom{z}{n} \right) = \sum_{m \in \NN^j} f_{nm} \binom{z}{m} \]
with $f_{nm} \in \Zp$, then
the matrix coefficients of $\tilde{f}$ must be
\[ \tilde{f}_{nm} = \alpha^{\cl{r \sum n}-\cl{s \sum m}} f_{nm} \,, \]
and the matrix coefficients of $\tilde{f}^*$ must be
\[ \tilde{f}^*_{nm} = \tilde{f}_{mn} = \alpha^{\cl{r \sum m}-\cl{s \sum n}} f_{mn} \,. \]
\end{proof}
\begin{lem} \label{lem:universal}
There exists $t_0 \in \RR^+$
so that for any $r,s \in \RR^+$, $j,k \in \NN$, and any
$\Zp$-module homomorphism
$f \colon \CC(\Zp^k,\Zp) \to \CC(\Zp^j,\Zp)$
that induces a continuous homomorphism
\[ \AA^{(p,r)}(\Zp^k,\Qp) \to \AA^{(p,s)}(\Zp^j,\Qp) \,,  \]
$f$ also induces continuous homomorphisms
\[ \tilde{f} \colon \Aar[rt](\Zp^k,A) \to \Aar[st](\Zp^j,A) \]
\[ \tilde{f}^* \colon \Dar[st](\Zp^j,A) \to \Dar[rt](\Zp^k,A) \]
for all $t \in (0,t_0)$.
\end{lem}
\begin{proof}
If the map $\tilde{f} \colon \Aar[rt](\Zp^k, A) \to \Aar[st](\Zp^j,A)$ exists,
then in the notation of the previous lemma, its matrix coefficients must be
given by
\[ \tilde{f}_{nm} = \alpha^{\cl{rt \sum n}-\cl{st \sum m}} f_{nm} \,. \]
Conversely, if there is a continuous map with these matrix coefficients,
then it is the desired map $\tilde{f}$.

The $\tilde{f}_{nm}$ are the matrix coefficients of a continuous map
if and only if the following two conditions are satisfied.
\begin{enumerate}
\item $\tilde{f}_{nm}$ are bounded.
\item For any fixed $n$, $\tilde{f}_{nm} \to 0$ as $\sum m \to \infty$.
\end{enumerate}
The terms with $r\sum n-s \sum m \ge 0$ are certainly bounded, so we only need
to worry about terms with $r \sum n-s \sum m < 0$.  There exists a positive integer $\ell$ so that $p^{\ell}/\alpha$ is
power bounded.  If the $p^{\cl{rt\ell \sum n}-\cl{st\ell \sum n}} f_{nm}$
(considered as elements of $\Qp$) are bounded
(resp.~go to zero as $\sum m \to \infty$), then the same will be true of
the $\alpha^{\cl{rt\sum n}-\cl{st\sum m}} f_{nm}$ (considered as elements of $A$).
So we may take $t_0 = \ell^{-1}$.

Similarly, if the map
$\tilde{f}^* \colon \Dar[st](\Zp^j,A) \to \Dar[rt](\Zp^k,A)$
exists, then its matrix coefficients satisfy $\tilde{f}^*_{mn} = \tilde{f}_{nm}$.
The $\tilde{f}^*_{mn}$ are the coefficients of a continuous map
if and only if condition (1) above and the following condition are satisfied.
\begin{itemize}
\item[(2$'$)] For any fixed $m$, $\tilde{f}_{nm} \to 0$ as $\sum n \to \infty$.
\end{itemize}
Since $f_{nm} \in \Zp$ and $\alpha^{\cl{rt\sum n}} \to 0$ as $\sum n \to \infty$,
condition (2$'$) will always be satisfied.
\end{proof}

\begin{prop} \label{prop:comp}
Let $g \colon \Zp^j \to \Zp^k$ be a (globally) analytic function.
For some $r_0 \in \RR^{+}$ depending on $\alpha$ but not on $g$, $j$, $k$, composition
with $g$ induces continuous $A$-linear homomorphisms
\[ \Aar(\Zp^k,A) \to \Aar[s](\Zp^j,A) \]
\[ \Dar[s](\Zp^j,A) \to \Dar(\Zp^k,A) \]
for all $s<r<r_0$.
\end{prop}
\begin{proof}
There are continuous maps
\[ \AA^{(p,1/(p-1))}(\Zp^k,\Qp) \to LA_0(\Zp^k,\Qp) \xrightarrow{g^*} LA_0(\Zp^j,\Qp) \to \AA^{(p,1/(p-1)-\epsilon)}(\Zp^j,\Qp) \]
for any $\epsilon \in (0,1/(p-1))$.
Applying Lemma \ref{lem:universal} then yields the desired result.
\end{proof}
If $j=1$, then the maps exist even if $r=s$.
We do not know if the same is true for $j>1$.
When $j=1$, one can prove existence by considering the norm on $LA_0$ defined in Lemma \ref{lem:la}
and using the fact that
$v_p(n!) - \sum_{i=1}^k v_p(m_i!) \ge \lfloor (n-\sum m)/p \rfloor$.
(The same idea will be used in the proof of Proposition \ref{prop:local}.)
However, for $j>1$,
$\sum_{i=1}^j v_p(n_i!) - \sum_{i=1}^k v_p(m_i!)$ can be zero for
arbitrarily large values of $\sum n - \sum m$.

\begin{prop} \label{prop:local}
Let $S$ be a set of coset representatives of $\Zp^k/p\Zp^k$.
The homeomorphism $\Zp^k \times S \isom \Zp^k$ defined by
$(z,s) \mapsto pz+s$
determines an isomorphism
\[ \CC(\Zp^k,A) \cong \CC(\Zp^k,A)^{\oplus p^k} \]
which induces isomorphisms
\[ \Aar(\Zp^k,A) \cong \Aar[pr](\Zp^k,A)^{\oplus p^k} \]
\[ \Dar(\Zp^k,A) \cong \Dar[pr](\Zp^k,A)^{\oplus p^k} \]
for all sufficiently small $r \in \RR^+$.
\end{prop}
\begin{proof}
First, consider the case $k=1$.
Applying Lemma \ref{lem:universal} along with translation invariance, we see
that
it is then enough to check that composition with the function
\[ g(z) = pz \]
defines a continuous homomorphism
\[ \AA^{(p,1/2p^2)}(\Zp,\Qp) \to \AA^{(p,1/2p)}(\Zp,\Qp) \]
and that composition with the function
\[ h(z) = \begin{cases}
z/p, & z \in p\Zp \\
0, & z \in \Zp^{\times}
\end{cases} \]
defines a continuous homomorphism
\[ \AA^{(p,1/2p)}(\Zp,\Qp) \to \AA^{(p,1/2p^2)}(\Zp,\Qp) \,. \]
Define $g_{nm}, h_{nm}$ by
\[ g \left( \binom{z}{n} \right) = \sum_{m=0}^{\infty} g_{nm} \binom{z}{m}, \quad
h \left( \binom{z}{n} \right) = \sum_{m=0}^{\infty} h_{nm} \binom{z}{m}\,. \]
By the same reasoning as in Lemma \ref{lem:universal},
we just need to verify that:
\begin{enumerate}
\item $v_p(g_{nm})-\frac{m}{2p}+\frac{n}{2p^2}$ is bounded below for $pm \ge n$.
\item For any $n$,  $v_p(g_{nm})-\frac{m}{2p}+\frac{n}{2p^2} \to \infty$ as $m \to \infty$.
\item $v_p(h_{nm})-\frac{m}{2p^2}+\frac{n}{2p}$ is bounded below for $m \ge pn$.
\item For any $n$,  $v_p(h_{nm})-\frac{m}{2p^2}+\frac{n}{2p} \to \infty$ as $m \to \infty$.
\end{enumerate}
Applying Lemma \ref{lem:la} gives
\[ v_p(g_{nm}) - v_p(m!) \ge -v_p(\fl{n/p}!) \,. \]
For $n \ge pm$ this implies
\[ v_p(g_{nm}) \ge \sum_{i=1}^{\infty} \left( \fl{m/p^i} - \fl{n/p^{i+1}} \right) \ge \fl{m/p-n/p^2} \,. \]
Similarly, Lemma \ref{lem:la} implies
\[ v_p(h_{nm}) \ge \sum_{i=1}^{\infty} \left( \fl{m/p^{i+1}}-\fl{n/p^i} \right) \ge \fl{m/p^2-n/p} \,. \]
This proves the case $k=1$.

We reduce the general case to the case $k=1$ as follows.
Since the above modules are all preserved by translation,
if the proposition is true for one choice of $S$, it is true
for any choice of $S$.  In particular, we may assume
$S$ is a product of $k$ sets of coset representatives of $\ZZ/p\ZZ$.
Then, since multiplication by $p$ does not
mix coordinates, the argument is essentially the same as in the $k=1$ case.
\end{proof}

\begin{lem} \label{lem:charan}
Any continuous homomorphism $\lambda \colon \Zp^k \to A^{\times}$ is in $\AA(\Zp^k,A)$.
\end{lem}
\begin{proof}
Lemma \ref{lem:mult} allows us to reduce to the one-dimensional case,
and Proposition \ref{prop:local} allows us to replace $\Zp^k$ with an open sub-lattice.
So it suffices to consider the case where $k=1$ and $(\lambda(1)-1)/\alpha$ is
topologically nilpotent.  In that case, since
\[ \lambda(z) = \sum_{n=0}^{\infty} \binom{z}{n} (\lambda(1)-1)^n, \]
$\lambda \in \Aar[1](\Zp,A)$.
\end{proof}

\begin{lem} \label{lem:inccc}
For any $0<s<r$, the inclusions
$\Aar[r](\Zp^k,A) \hookrightarrow \Aar[s](\Zp^k,A)$ and $\Dar[s] \hookrightarrow \Dar[r](\Zp^k,A)$ are
completely continuous.
\end{lem}
\begin{proof}
In the orthonormal bases of Lemma \ref{lem:laconcrete},
these inclusions are represented by diagonal matrices with
diagonal entries of the form $\alpha^{\fl{r \sum n}-\fl{s \sum n}}$.
As $\sum n \to \infty$, the entries go to zero.
\end{proof}

\subsection{Gluing}

Propositions \ref{prop:comp} and \ref{prop:local} show that it makes sense to define
locally analytic functions and distributions on
arbitrary locally $\Qp$-analytic manifolds by gluing.
\begin{defn}
Let $k$ be a nonnegative integer, and let $X$ be a locally $\Qp$-analytic manifold of dimension
$k$.  Choose a decomposition $X = \bigsqcup_{i \in I} X_i$
for some index set $I$, and choose an identification of each $X_i$ with $\Zp^k$.
We define
\begin{eqnarray*}
\AA(X,A) & = & \prod_{i \in I} \AA(X_i,A) \\
\AA_c(X,A) & = & \bigoplus_{i \in I} \AA(X_i,A) \\
\DD(X,A) & = & \bigoplus_{i \in I} \DD(X_i,A) \\
\DD_c(X,A) & = & \prod_{i \in I} \DD(X_i,A) \,.
\end{eqnarray*}
\end{defn}
By Propositions \ref{prop:comp} and \ref{prop:local}, the above definitions
do not depend on the choice of decomposition.

\begin{thm} \label{thm:la}
The modules $\AA(X,A)$, $\DD(X,A)$, $\AA_c(X,A)$, and $\DD_c(X,A)$ satisfy the following
properties:
\begin{enumerate}
\item $\AA(X,A)$ is ring.
\item If $g \colon X \to Y$ is a locally analytic map, then composition with
$g$ induces homomorphisms $\AA(Y,A) \to \AA(X,A)$ and
$\DD(X,A) \to \DD(Y,A)$.
\item The functors $U \mapsto \AA(U,A)$ and $U \mapsto \DD_c(U,A)$ are sheaves on $X$.
\item If $X$ has the structure of a finitely generated $\Zp$-module, then any continuous
group homomorphism $X \to A^{\times}$ is in $\AA(X,A)$.
\end{enumerate}
\end{thm}
\begin{proof}
The claims all follow immediately from the results of section \ref{sec:laproperties}.
\end{proof}

\subsection{Geometric interpretation of distributions}

The modules of locally analytic distributions have an alternative
interpretation as rings of sections of adic spaces.  This interpretation will not
be used elsewhere in the paper, but it gives further evidence that our definition of distributions
is reasonable.  For background on adic spaces, see \cite{Hub93,Hub94,Hub96} or \cite[\S 2-5]{berkeley}.

Let $D = \Spa(\Zp \db{\Zp^k},\Zp \db{\Zp^k})$.
Suppose that the Tate algebra
$A \ang{T_1,\dotsc,T_n}$ is sheafy for each nonnegative integer $n$.
Let $A^+$ be an open and
integrally closed subring of $A$.
Let $Y = D \times_{\Spa(\Zp,\Zp)} \Spa(A,A^+)$.
We can construct $Y$ as follows.
There is an isomorphism $\Zp \db{T_1,\dotsc,T_k} \cong \Zp \db{\Zp^k}$
that sends $T_i \mapsto \Delta_{e_i}$, where the $e_i$ form a basis of $\Zp^k$;
this isomorphism is known as the multivariable Amice transform.
For any positive rational $r=m/n$,
let $B_r=A\ang{T_1,\dotsc,T_k,T_1^n/\alpha^m,\dotsc,T_k^n/\alpha^m}$,
and let $B_r^+$ be the normal closure of
$A^+\ang{T_1,\dotsc,T_n,T_1^n/\alpha^m,\dotsc,T_n^n/\alpha^m}$ in $B_r$.
Then $Y$ is formed by gluing the affinoids $Y_r \colonequals \Spa(B_r,B_r^+)$.

There are canonical isomorphisms
\begin{IEEEeqnarray*}{rCl}
\Hom_{\Zp}(\CC(\Zp^k,\Zp),\Zp) & \cong & \OO_D(D) \\
\DD(\Zp^k,A) & \cong & \OO_Y(Y) \\
\Dar(\Zp^k,A) & \cong & \OO_Y(Y_r) \quad \forall r \in \QQ^+ \,.
\end{IEEEeqnarray*}

\section{Overconvergent cohomology} \label{sec:automorphic}

Now we use the modules constructed in section \ref{sec:la}
to define overconvergent cohomology.
We mostly repeat the setup of \cite[\S 3-4]{Urb11}; see also
\cite{AS08} and \cite[\S 2-3]{Han15}.

\subsection{Locally symmetric spaces} \label{sec:lss}

Let $\AD$ (resp.~$\AD_f$, $\AD_f^p$) be the ring of adeles
(resp.~finite adeles, finite adeles away from $p$) of $\QQ$.

Let $G$ be a connected reductive algebraic group over $\QQ$.
We will assume that $G(\Qp)$ is quasisplit.  Let $B,T,N,N^-$ be compatible
choices of a Borel subgroup, maximal torus, maximal unipotent
subgroup, and opposite unipotent subgroup, respectively, of $G(\Qp)$.

We will need some results from \cite{bruhat-tits-1}.  Note that
$G_{\Qp}$ admits a valued root datum (``donn\'ee radicielle valu\'ee'')
by \cite[4.2.3 Th\'eor\`eme]{bruhat-tits-2}.

Let $I$ be an Iwahori subgroup of $G(\Qp)$ compatible with $B$ (see for
example \cite[\S 6.5]{bruhat-tits-1}; note that this reference denotes
the Iwahori by $B$).
Then $I$ admits a factorization $I=N_0 T_0 N_0^{-}$, where
$N_0^{-} = N^- \cap I$, $T_0 = T \cap I$, $N_0 = N \cap I$.
Let $K^p$ be an open compact subgroup
of $\AD_f^p$, and let $K = K^p I$.  We assume that $K$ is neat; see
Definition \ref{neat} below.
Let $G_{\infty}^+$ be the identity component of $G(\RR)$, and
let $K_{\infty}$ be a maximal compact modulo center subgroup of $G_{\infty}^+$.
Let $Z_G$ be the center of $G$.

The space
\[\XX \colonequals G(\AD)/K^p G_{\infty}^+ \]
may be considered as a locally
$\Qp$-analytic manifold.
Let $A$ be a complete Noetherian Tate $\Zp$-algebra.
In section \ref{sec:la}, we defined
the module $\DD_c(\mathcal{X},A)$ of ``locally analytic'' compactly
supported $A$-valued distributions on $\XX$.

Let
$\lambda \colon T_0 \to A^{\times}$ be a continuous homomorphism.
By Lemma \ref{lem:charan}, $\lambda \in \AA(T_0, A)$.
We will assume that $\ker \lambda$ contains $(Z_G(\QQ) K^p G_{\infty}^+ \cap T_0)$.
We define
$\DD_{c,\lambda}(\XX,A)$ to be the quotient of
$\DD_c(\mathcal{X},A)$ obtained by constraining
right translation by $N_0^-$ to act by the identity,
right translation by $T_0$ to act by $\lambda$,
and translation by $Z_G(\QQ)$ to act by the identity.

The group $\Gadq$ acts on $\DD_{c,\lambda}(\XX,A)$ by left
translation.  Moreover, $\DD_{c,\lambda}(\XX,A)$ is
a direct sum of modules induced from much smaller
subgroups of $\Gadq$.
We can write $G(\AD)$ as a finite union
\[ G(\AD) = \bigsqcup_{i} G(\QQ) g_i G_{\infty}^+ K \,. \]
Let $\Gamma_i$ be the image of $g_i G_{\infty}^+ K g_i^{-1} \cap G(\QQ)$ in $\Gadq$.
Then
\[ \DD_{c,\lambda}(\XX,A) \cong \bigoplus_i \Ind_{\Gamma_i}^{\Gadq} \DD_{\lambda}(g_i I,A) \]
where $\DD_{\lambda}(g_i I,A)$ is the quotient of $\DD(g_i I,A)$ obtained
by constraining right translation by $N_0^-$ to act as the identity
and right translation by $T_0$ to act as $\lambda$.  Here $\Gamma_i$ acts
on $\DD_{\lambda}(g_i I,A)$ by left translation.

The existence of the Iwahori factorization implies that the
map $N_0 \to g_i I$ given by $n \mapsto g_i n$ induces an isomorphism
of $A$-modules
\[ \DD(N_0,A) \isom \DD_{\lambda}(g_i I,A) \,. \]
This identification induces a $\Gamma_i$-action on $\DD(N_0,A)$, which
can be described as follows.
Any $x \in I$ has an Iwahori factorization $x = \bn(x) \bt(x) \bn^-(x)$
with $\bn(x) \in N_0$, $\bt(x) \in T_0$, $\bn^-(x) \in N_0^-$, and the functions
$\bn$, $\bt$, and $\bn^-$ are analytic.
The action of $\Gamma_i$ on $\DD(N_0,A)$ is given by
\[ \gamma \cdot [x] = \lambda(\bt(g_i^{-1} \gamma g_i x)) [\bn(g_i^{-1} \gamma g_i x)] \]
for $\gamma \in \Gamma_i$, $x \in N_0$.  Here $[x]$ denotes the
Dirac delta distribution supported at $x$.

Now consider the locally symmetric space
\[ S_G(K) \colonequals G(\QQ) \backslash G(\AD) / K_{\infty} K \,. \]
Then $S_G(K) \cong \bigsqcup_i \mathcal{Y}_i$ where
\[ \mathcal{Y}_i \colonequals \Gamma_i \backslash G^+_{\infty}/K_{\infty} \,. \]
\begin{defn} \label{neat}
We say that $K$ is \emph{neat} if all of the $\Gamma_i$ are torsionfree.
\end{defn}
As mentioned above, we assume that $K^p$ has been chosen so that $K$ is neat.
Then each $\mathcal{Y}_i$
is a manifold with fundamental group $\Gamma_i$.

The manifold $S_G(K)$ has a Borel-Serre compactification $\overline{S_G(K)}$,
which is homotopy equivalent to $S_G(K)$.
Any finite triangulation of $\overline{S_G(K)}$ determines
a resolution
\[ 0 \to C_{d}(\Gamma_i) \to \dotsb \to C_1(\Gamma_i) \to C_0(\Gamma_i) \to \ZZ \to 0 \]
where the $C_j(\Gamma_i)$ are free $\ZZ[\Gamma_i]$-modules of finite rank and $d$ is
the dimension of $S_G(K)$.
We define a complex $C^{\bullet}_{\lambda}$ by
\begin{equation} C^j_{\lambda} \colonequals \bigoplus_i \Hom_{\Gamma_i} (C_j(\Gamma_i),\DD_{\lambda}(g_i I,A)) \,. \label{complexdef}
\end{equation}
Then
\[ R\Gamma^{\bullet}(\Gadq,\DD_{c,\lambda}(\XX,A))
\cong \bigoplus_i R\Gamma^{\bullet}(\Gamma_i, \DD_{\lambda}(g_i I, A))
\cong C^{\bullet}_{\lambda} \]
in the derived category of $A$-modules.

\subsection{Hecke action} \label{sec:hecke}

We choose a projective resolution
\[ \dotsb \to C_1(\Gadq) \to C_0(\Gadq) \to \ZZ \to 0 \]
of $\ZZ$ as a $\Gadq$-module as well as maps of
complexes of $\Gamma_i$-modules
$C_{\bullet}(\Gamma_i) \to C_{\bullet}(\Gadq)$ and
$C_{\bullet}(\Gadq) \to C_{\bullet}(\Gamma_i)$ that are homotopy
inverses of each other.  Then any
$f \in \End_{\Gadq}(\DD_{c,\lambda}(\mathcal{X},A))$ defines an operator
$[f] \in \End \left(C^{\bullet}_{\lambda} \right)$
by
\[ \begin{split}
C^j_{\lambda} \to & \bigoplus_i \Hom_{\Gamma_i} (C_j(\Gadq),\DD_{\lambda}(g_i I,A)) \\
\isom & \Hom_{\Gadq} (C_j(\Gadq),\DD_{c,\lambda}(\mathcal{X},A)) \\
\xrightarrow{f} & \Hom_{\Gadq} (C_j(\Gadq),\DD_{c,\lambda}(\mathcal{X},A)) \\
\isom & \bigoplus_i \Hom_{\Gamma_i} (C_j(\Gadq),\DD_{\lambda} (g_i I,A)) \\
\to & C^j_{\lambda} \,.
\end{split} \]
For any $f,g$, $[f][g]$ is homotopy equivalent to $[fg]$.

For any $g \in G(\Afp)$, the double coset operator $K^p g K^p$ acts on
$\DD_{c,\lambda}$ and determines a Hecke operator
$[K^p g K^p]$ on $C^{\bullet}_{\lambda}$.

Let
\[ T^- \colonequals \left\{ t \in T \middle| t^{-1} N_0^- t \subseteq N_0^- \right\} \,. \]
For $t \in T^-$, the double coset operator
$N_0^- t N_0^-$ acts on $\DD_{c,\lambda}$ and determines
an operator $[N_0^- t N_0^-]$ on $C^{\bullet}_{\lambda}$.
We will sometimes denote this operator
by $u_t$.

\begin{rmk}
Our definition of the Hecke operators at $p$ differs slightly from that of
previous references on overconvergent cohomology, which made use of a choice of
``right $*$-action''.  Our definition is instead meant to be analogous to
the one used in Emerton's theory of completed cohomology \cite{Eme06J,Eme06C}.
The two approaches will yield the same eigenvariety.
The only essential difference between the approaches is that, to
define a ``right $*$-action'', one chooses a splitting of
$0 \to T_0 \to T \to T/T_0 \to 0$, and then uses this splitting
to twist the Hecke operators so that $T_0$ acts trivially.
\end{rmk}

Let $S$ be the set of finite places at which $K^p$ is not maximal hyperspecial.
Let $\AD_f^{p,S}$ be the adeles away from $p$ and $S$, and let
$K^{p,S}$ be the image of $K^{p}$ in $\AD_f^{p,S}$.
We define the Hecke algebra
\[ \HH_G \colonequals C^{\infty}_c \left( K^{p,S} \backslash G(\AD_f^{p,S})/K^{p,S} \times N_0^- \backslash N_0^- T^- N_0^- / N_0^-, \Zp \right) \,. \]

\subsection{Topological properties of Hecke operators} \label{sec:hecketop}

In order to apply the spectral theory introduced in section \ref{sec:spectral},
we will need to choose a particular description of $C^{\bullet}_{\lambda}$
as a limit of complexes of projective modules.
The logarithm induces a bijection between $N_0$ and a finite free $\Zp$-module;
we use this bijection to define a coordinate chart on $N_0$.
This chart allows us to define the projective modules
$\Dar(N_0,A)$ for some arbitrarily chosen topologically nilpotent unit $\alpha \in A$.
Define
\[ C^i_{\lambda,\alpha,r} \colonequals \bigoplus_j \Hom_{\Gamma_j}(C_i(\Gamma_j),\Dar(N_0,A)) \,. \]
\begin{lem} \label{lem:projdiff}
For all sufficiently small $r$ and all $\epsilon>0$,
the differential $d \colon C^{i+1}_\lambda \to C^i_{\lambda}$ extends to a map
$C^{i+1}_{\lambda,\alpha,r} \to C^i_{\lambda,\alpha,r+\epsilon}$.
\end{lem}
\begin{proof}
It is enough to check that for sufficiently small $r$
and all $\epsilon>0$, left translation by any $\gamma \in \Gamma_i$ maps
$\Dar(N_0,A)$ into $\Dar[r+\epsilon](N_0,A)$.
This follows from the description of the action in section \ref{sec:lss} along
with Lemmas \ref{lem:mult} and \ref{lem:charan} and Proposition \ref{prop:comp}.
\end{proof}

If $\underline{r}=(r_0,\dotsc,r_{d})$ is chosen such that the differentials
$C^{i+1}_{\lambda,\alpha,r_{i+1}} \to C^i_{\lambda,\alpha,r_i}$ are defined, then we denote
the corresponding complex by $C^{\bullet}_{\lambda,\alpha,\underline{r}}$.

Choose some $t \in T^-$ such that $t^{-1} N_0 t \subset N_0^p$.  Let
$\HH'_G$ be the ideal of $\HH_G$ generated by $u_t$.

\begin{lem}
There exists $r_0 \in \RR^+$ so that for all
$r \in (0,r_0)$, $\epsilon \in \RR^+$, and $f \in \HH'_G$,
$f$ determines a continuous map
$C^i_{\lambda,\alpha,r} \to C^i_{\lambda,\alpha,r/p+\epsilon}$,
and hence $f$ determines a completely continuous map
$C^i_{\lambda,\alpha,r} \to C^i_{\lambda,\alpha,r}$.
\end{lem}
\begin{proof}
We can show that Hecke operators away from $p$ map $C^i_{\lambda,\alpha,r}$
into $C^i_{\lambda,\alpha,r+\epsilon}$ using essentially the
same argument as in Lemma \ref{lem:projdiff}.  It remains to show
that $u_t$ maps $C^i_{\lambda,\alpha,r}$ into
$C^i_{\lambda,\alpha,r/p+\epsilon}$.
The action of $u_t$ can be built from functions of the form
\[ [x] \mapsto \lambda(\bt(\iota(x))) [\bn(\iota(x))]  \]
where $\iota(x)$ takes the form
\[ \iota(x) = h t^{-1} \bn(n^- x) \bt(n^- x) t \]
for some $n^- \in N_0^-$, $h \in I$.
(See for example \cite[Lemma 4.2.19]{Eme06J}.)
In particular, $\bn(\iota(x))$ belongs to a
single right coset of $t^{-1} N_0 t \subset N_0^p$.
The argument proceeds as before, except that we also need to use
Proposition \ref{prop:local} and Lemma \ref{lem:inccc}.
\end{proof}

\subsection{Characteristic power series} \label{sec:auto-fredholm}
For any $f \in \HH'_G$, we define
the power series
\[ \det\left(1-Xf \middle| C^{\bullet}_{\lambda} \right) \colonequals
\det\left(1-Xf \middle| C^{\bullet}_{\lambda,\alpha,\ur} \right) \]
for any $\alpha,\ur$ for which the complex $C^{\bullet}_{\lambda,\alpha,\ur}$
is defined and the $u_t$ operator is completely continuous.
Choosing a different $\alpha$ and $\ur$ conjugates the matrix
of $f$ by a diagonal matrix, so the power series does not depend on
them.

Similarly, we define $\det \left(1-Xf \middle| C^{i}_{\lambda} \right)
\colonequals \det \left(1-Xf \middle| C^{i}_{\lambda,\alpha,\ur} \right)$.
Consider the Fredholm series
\[ P_+(X) \colonequals \prod_{i=0}^{d} \det \left( 1-Xu_t \middle| C^i_{\lambda} \right) \,. \]
Suppose that $P_+(X)$ factors as $Q_+(X) S_+(X)$, with
$Q_+(X) \in A[X]$, $S_+(X) \in A \fs{X}$, that $Q_+(X)$ and $S_+(X)$ are
relatively prime, and that the leading coefficient of $Q_+(X)$ is invertible.
Let $Q_+^*(X)=X^{\deg Q_+} Q_+(X^{-1})$.  By
\cite[Th\'eor\`eme B.2]{AIP15}, there is a decomposition
$C_{\lambda,\alpha,\ur}^{\bullet} = N^{\bullet}_{\alpha,\ur} \oplus F^{\bullet}_{\alpha,\ur}$,
where
$Q_+^*(u_t)$ annihilates $N^{\bullet}_{\alpha,\ur}$ and acts invertibly on
$F^{\bullet}_{\alpha,\ur}$, and the $N^i_{\alpha,\ur}$ are
finitely generated and projective.
\begin{lem}
For any $\alpha,\alpha'$ and $\ur,\ur'$ such that
$N^{\bullet}_{\alpha,\ur}$ and
$N^{\bullet}_{\alpha',\ur'}$ are defined, they are canonically isomorphic.
\end{lem}
\begin{proof}
Choose $\ur''$ so that $C^{\bullet}_{\lambda,\alpha,\ur''}$ injects into
$C^{\bullet}_{\lambda,\alpha,\ur}$ and $C^{\bullet}_{\lambda,\alpha',\ur'}$.
The operator $1 - \frac{Q_+^*(u_t)}{Q_+^*(0)}$ acts as the identity on
$N^{\bullet}_{\alpha,\ur}$,
and for sufficiently large $n$,
$\left(1 - \frac{Q_+^*(u_t)}{Q_+^*(0)} \right)^n$
factors through $N^{\bullet}_{\alpha,\ur''}$.  So we get a canonical isomorphism
$N^{\bullet}_{\alpha,\ur} \cong N^{\bullet}_{\alpha,\ur''}$, and similarly
there is a canonical isomorphism $N^{\bullet}_{\alpha',\ur'} \cong N^{\bullet}_{\alpha,\ur''}$.
\end{proof}
\begin{cor}
There is a decomposition $C_{\lambda}^{\bullet}=N^{\bullet} \oplus F^{\bullet}$,
where $Q_+^*(u_t)$ annihilates $N^{\bullet}$ and acts invertibly on
$F^{\bullet}$,
and the $N^i$ are finitely generated and projective.
\end{cor}

\section{Eisenstein and cuspidal contributions to characteristic power series}
\label{sec:eisenstein}

\subsection{Preliminaries}

In this section, we will write $C^{\bullet}_{G,K^p,\lambda}$ for
$C^{\bullet}_{\lambda}$ to make it clear which group we are considering.
We will also assume that $G(\RR)$ has discrete series (i.e.~$G(\RR)$ admits representations
with essentially square integrable matrix coefficients, or equivalently
$G(\RR)$ has a maximal torus that is compact modulo $Z_G(\RR)$),
since otherwise Urban's eigenvariety will be empty.

In order to construct Urban's eigenvariety, we need the characteristic power
series of the Hecke operators to be Fredholm series.
However, the power series
$\det \left(1-Xf \middle| C^{\bullet}_{G,K^p,\lambda} \right)$
includes contributions from both cusp forms and Eisenstein series,
and the Eisenstein contribution is generally only a ratio of Fredholm series.
We will now define a complex $C^{\bullet}_{G,K^p,\lambda,\cusp}$
whose characteristic power series only includes contributions from cusp forms.
(This complex will only be useful for defining characteristic power series; we
make no attempt to remove the Eisenstein series from the cohomology.)

We will mostly follow \cite[\S 4.6]{Urb11}.
However, there is an error in the handling of the Eisenstein series in
\cite{Urb11} that we will need to correct.  The region of
convergence of an Eisenstein series is generally not a union of Weyl
chambers.  (For example, $Sp(6)$ has two conjugacy classes of parabolic
subgroups whose Levis are isomorphic to $GL(2) \times GL(1)$.  The
region of convergence of Eisenstein series coming from these parabolics
contains one or two full Weyl chambers and fractions of
three others.)
Consequently, the set $\mathcal{W}^M_{\operatorname{Eis}}$ defined in
\cite{Urb11}
should depend on the weight of the Eisenstein series.  A more careful argument
is therefore needed to show that character distribution $I^{cl}_{G,0}(f,\mu)$
has a unique $p$-adic interpolation.  In fact, it appears that the character
distribution of Eisenstein series coming from a single parabolic subgroup will
generally \emph{not} have a unique interpolation.
We will show, however, that the sum of distributions coming
from parabolic subgroups that have a common Levi will have a unique
interpolation.

Let $W_G$ denote the Weyl group of $G$.  Let
$\Phi_G$, $\Phi_G^{\vee}$ denote the set of roots and coroots, respectively,
of the pair ($G_{\Qp}$, $T$), where $T$ is the torus
chosen in section \ref{sec:automorphic}.  Let $\Phi_G^+$ (resp. $\Phi_G^-$)
denote the subset of roots that
are positive (resp.\ negative) with respect to $B$, and we make a similar definition for coroots.
Let $\rho$ denote half the sum of the roots in $\Phi_G^+$.

Let $F$ be a finite extension of $\mathbb{Q}_p$.  We say that
$\mu \colon T_0 \to F^{\times}$ is an algebraic weight if it can be
extended to a homomorphism of algebraic groups $T_{F} \to (\mathbb{G}_{m})_F$.
We say that an algebraic weight $\mu$ is dominant (resp.\ regular dominant)
if $\ang{ \alpha^{\vee}, \mu } \ge 0$ (resp. $>0$) for all
$\alpha^{\vee} \in \Phi_G^{\vee+}$.

Suppose that $\mu$ is dominant.
Then $\DD_{\mu}(g_i I,F)$ has
a (nonzero) quotient that is a finite-dimensional $F$-vector space.
We will write $L^G_{\mu}$ for the corresponding local system
on either $S_G(K)$ or $\overline{S_G(K)}$.  (In \cite{Urb11}, this local
system is denoted $\mathbb{V}_{\lambda}^{\vee}$.)

\begin{lem} \label{lem:interp}
Let $f = u_t \otimes f^p \in \HH_G'$, and let $\mu \colon T_0 \to F^{\times}$ be
an algebraic dominant weight.
Then
\[ \det \left(1-Xf|C^{\bullet}_{G,K^p,\mu} \right) \equiv
\det \left(1-Xf|H^{\bullet}(S_G(K),L^G_{\mu}) \right)
\pmod {\OO_{F} \db{N(\mu,t)X}} \]
where
\[ N(\mu,t) \colonequals \inf_{w \in W_G \setminus \{\id\}} |t^{(w-1)(\mu+\rho)}|_p \,. \]
\end{lem}
\begin{proof}
For the degree 1 term, this is \cite[Lemma 4.5.2]{Urb11}.  The argument
used there also works for higher degree terms.
\end{proof}

In section \ref{sec:eigenvariety}, we will consider a family of weights having the property
that for any $n \in \NN$, the set of points corresponding to regular dominant weights $\mu$
satisfying $p^n \mid N(\mu,t)$ is Zariski dense.  The characteristic power series for the whole
family can then be determined from the $\det(1-Xf | H^{\bullet}(S_G(K),L^G_{\mu}))$.

If $\mu$ is regular dominant, then the cuspidal subspace of $H^{i}(S_G(K),L^G_{\mu})$
is the interior cohomology $H^i_!(S_G(K),L^G_{\mu})$ \cite[\S 5.3]{LS04},
and furthermore (since we assume $G(\RR)$ has discrete series)
the interior cohomology is nonzero only in the middle degree \cite[Theorem III.5.1]{BW00}.
Hence either $\det(1-Xf|H^{\bullet}_!(S_G(K),L^G_{\mu}))$ or its reciprocal is a polynomial.

Our goal is to prove a version of Lemma \ref{lem:interp} in which
$C^{\bullet}_{G,K^p,\mu}$ is replaced by a complex
$C^{\bullet}_{G,K^p,\mu,\cusp}$ that we will define,
and $H^{\bullet}(S_G(K),L^G_{\mu})$ is replaced by
$H^{\bullet}_!(S_G(K),L^G_{\mu})$.

\subsection{Cohomology of the Borel-Serre boundary} \label{sub:bsbdy}

Eisenstein series arise from the Borel-Serre boundary
$\partial S_G(K) \colonequals \overline{S_G(K)} \setminus S_G(K)$ of $S_G(K)$.
The boundary has a stratification by locally symmetric spaces
of parabolic subgroups of $G$.

We warn the reader that the Borel-Serre compactification $\overline{S_G(K)}$ is slightly strange.
When constructing a locally symmetric space, one usually takes a quotient by
the identity component of either $Z_G(\RR)$ or $A_G(\RR)$, where $A_G$ is the $\QQ$-split part of $Z_G$.
In order to construct Urban's eigenvariety, we need to choose the former option, but the Borel-Serre
compactification behaves better with respect to the latter.  Consequently, if $M$ is a Levi
subgroup of $G$, then the locally symmetric space for $M$ should be constructed by taking
a quotient by the identity component of $Z_G(\RR) A_M(\RR)$ rather than that of $Z_M(\RR)$.
However, it will turn out that we only need to consider Levi subgroups for which the two quotients
are the same; see section \ref{sub:img} for more details.

Let $P$ be a parabolic subgroup of $G$, let $N$ be the maximal unipotent subgroup of $P$,
and let $M=P/N$ be its Levi quotient.
Let $K^p_P = K^p \cap P(\Afp)$, $K_{P,p}=I \cap P(\Qp)$, $K_P = K^p_P K_{P,p}$.
We can define a locally symmetric space $S_P(K_P)$, and there is a locally
closed immersion
\[ \iota \colon S_P(K_P) \to \overline{S_G(K)} \,. \]
If $P'$ is another parabolic subgroup of $G$,
then $S_P(K_P)$ and $S_{P'}(K_{P'})$ will have the same image in
$\overline{S_G(K)}$ if and only if $P(\Af)$ and $P'(\Af)$ are conjugate
by an element of $K^p I$.

Let $K^p_M$, $I_M$ be the images of $K^p_P$, $K_{P,p}$ in $M(\Afp)$, $M(\Qp)$, respectively.
The group $I_M$ is an Iwahori subgroup of $M$.  Let $K_M = I_M K^p_M$.  The locally symmetric space
$S_P(K_P)$ is a nilmanifold bundle over $S_M(K_M)$.  Let
\[ \pi \colon S_P(K_P) \to S_M(K_M) \]
denote the projection.

We can relate $R \pi_* \iota^* L^G_{\mu}$ to local systems on $S_M(K_M)$
using the Kostant decomposition \cite[Theorem III.3.1]{BW00}.
To define the local systems on $S_M(K_M)$, we first need to choose a
quasisplit torus $T_M$ of $M$.
The parabolic subgroup $P_{\Qp}$ contains a conjugate of $B$.
There is a decomposition $G(\Qp) = I N_G(S)(\Qp) B(\Qp)$, where $N_G(S)$
is the normalizer of the maximal split subtorus $S$ of $T$; this follows from
\cite[\S 4.2.5, Th\'eor\`eme 5.1.3]{bruhat-tits-1} as well as from \cite[Proposition 7.3.1]{bruhat-tits-1}.
So $iwBw^{-1}i^{-1} \subseteq P_{\Qp}$ for some $i \in I$, $w \in N_G(S)(\Qp)$.
We choose $i$ and $w$ to minimize the length of the image of $w$ in the Weyl group $W_G$.
Let $T_M$ be the image of $iwTw^{-1} i^{-1}$ in $M_{\Qp}$.
The obvious isomorphism $T \isom T_M$ determines a length-preserving
injection of Weyl groups $W_M \hookrightarrow W_G$.  Let $W^M$ denote a set
of minimal length coset representatives of $W_M \backslash W_G$.

We have the following isomorphism in the derived category of
constructible sheaves on $S_M(K_M)$.
\begin{equation} \label{eq:kostant}
R \pi_* \iota^* L^G_{\mu} \cong \bigoplus_{w' \in W^M} L^M_{w^{-1}(w'(\mu+\rho)+\rho-2\rho_M)}[l(w')-\dim N]
\end{equation}
Here $l(w')$ denotes the length of $w'$.  To see that the splitting exists
in the derived category and not just at the level of cohomology, we observe
that the $L^M_{w^{-1} (w' (\mu+\rho) + \rho-2\rho_M)}$ have distinct central
characters.
\begin{rmk}
A derivation of the dual Kostant decomposition from the usual
Kostant decomposition is given in \cite[\S 1.4.1]{Urb11}.  However,
the final formula is
incorrect because the second equality in formula (7) is not true.
\end{rmk}

\subsection{Hecke action}

We will now define an action of $\HH_G$ on the cohomology of $S_M(K_M)$
by constructing a homomorphism $\HH_G \to \HH_M$.  The map
$R \pi_* \iota^*$ will be equivariant for this action.

As explained in \cite[Corollary 4.6.3]{Urb11}, for any summand of
\eqref{eq:kostant} with $w \ne w'$,
the Hecke eigenvalues of $u_t \in \HH'_G$ acting on the cohomology of
this summand
will be divisible by $N(\mu,t)$.
Since our goal is to prove a cuspidal analogue of Lemma \ref{lem:interp},
we may ignore these summands and just consider the one with $w = w'$.
We are therefore only interested in the local system
\[ L^M_{w^{-1} (w(\mu+\rho)+\rho-2\rho_M)} = L^M_{\mu+\rho+w^{-1}(\rho-2
\rho_M)} \,. \]

Our definition of the homomorphism $\HH_G \to \HH_M$ will be the same as that
of \cite[4.1.8]{Urb11}, except that
our convention for the Hecke operators makes some normalization factors
disappear.  The Hecke algebra $\HH_G$ is generated by operators of the form
$u_t$ for $t \in T^{-}$ and $[K_v g K_v]$ for $v \notin S$, $g \in G(\QQ_v)$.
Let $u_t \in \HH_G$ act as $t^{\rho+w^{-1}(\rho-2\rho_M)} u_t \in \HH_M$.
The double coset $K_v g K_v$ decomposes as a finite union
$\bigsqcup_j K_{M,v} p_j K_v$ with $p_j \in P(\QQ_v)$.
Let $[K_v g K_v]$ act as $\sum_j [K_{M,v} m_j K_{M,v}]$,
where $m_j$ is the image of $p_j$ in $M(\QQ_v)$.

\begin{lem}
The homomorphism $\HH_G \to \HH_M$ defined above makes the map
\[ R \pi_* \iota^* \colon H^{\bullet}(S_G(K), L^G_{\mu}) \to H^{\bullet}(S_M(K_M),L^M_{\mu+\rho+w^{-1}(\rho-2\rho_M)})[l(w)-\dim N]  \]
$\HH_G$-equivariant.
\end{lem}
\begin{proof}
The argument is essentially the same as that of \cite[4.1.8, 4.6.1--3]{Urb11}.
\end{proof}

\subsection{Image of the map $R \pi_* \iota^*$} \label{sub:img}

To simplify some of the analysis that follows, we will observe that some
Levis have
$H^{\bullet}(S_M(K_M),L^M_{\mu})=0$ for a Zariski dense subset
of weights $\mu$, and hence they cannot contribute to the characteristic power
series.
The Levi $M$ can have a nonzero contribution only if the following conditions
hold (see \cite[Theorem 4.7.3(ii)$'$]{Urb11}):
\begin{enumerate}
\item $M(\RR)$ has discrete series. \label{disc}
\item The center $Z_M$ of $M$ is generated by its maximal split subgroup, its
maximal compact subgroup, and $Z_G$. \label{split}
\end{enumerate}
We will call $M$ \emph{relevant} if it satisfies the above two conditions.
We will say that a parabolic subgroup of $G$ is relevant if its
Levi quotient is relevant.

Now assume that $M$ is relevant.
We will define an involution
\[ \theta \colon X_*(T_M/Z_G) \to X_*(T_M/Z_G) \,. \]
To do this, we first decompose $X_*(T_M/Z_G) \otimes \QQ$ into several pieces.
We have
\[ X_*(T_M/Z_G) \otimes \QQ \cong \left( X_*(Z_M/Z_G) \oplus X_*(T_M \cap M^{\der}) \right) \otimes \QQ \,. \]
There is an action of $\Gal(\bar{\QQ}/\QQ)$ on $X_*(Z_M / Z_G)$.
This representation has open kernel, so it becomes semisimple
after tensoring with $\QQ$.
We define $\theta$ to be the operator that acts as $1$ on the isotypic
component of the trivial representation and as $-1$
on its orthogonal complement and on $X_*(T_M \cap M^{\der})$.
Although it is not immediately obvious that $\theta$
preserves the lattice $X_*(T_M/Z_G) \subset X_*(T_M/Z_G) \otimes \QQ$, the following alternative
description of $\theta$ will show that it does.

Let $T_M'$ be a maximal torus of $M_{\RR}$ that is compact modulo $(Z_M)_{\RR}$.
Such a torus exists by assumption \eqref{disc}.
If $C$ is an algebraically closed field equipped with inclusions
$\RR \hookrightarrow C$ and $\Qp \hookrightarrow C$,
then the tori $(T_M)_C$ and $(T_M')_C$ are conjugate in $M_C$.
Each way of expressing $(T_M')_C$ as a conjugate of $(T_M)_C$ determines
an isomorphism $X_*(T_M / Z_G) \simeq X_*(T_M' / Z_G)$.
We claim that under any such isomorphism, the action of complex conjugation on
$X_*(T_M' / Z_G)$ induces the involution $\theta$ on $X_*(T_M / Z_G)$.
Indeed, since $T_M'$ is compact modulo center, complex conjugation acts as $-1$
on $X_*(T_M' \cap M^{\der})$, and assumption \eqref{split} guarantees that
any element of $X_*(Z_M / Z_G)$ that is fixed by complex conjugation is
fixed by $\Gal(\bar{\QQ}/\QQ)$.

By \cite[\S 3.2]{LS04}, the image of
\[ R\pi_* \iota^* \colon H^{\bullet}(S_G(K),L^G_{\mu}) \to H^{\bullet}(S_M(K_M),L^{M}_{\mu+\rho+w^{-1}(\rho-2\rho_M)})[l(w)-\dim N]  \]
can have nonzero intersection with the cuspidal part
\[ H^{\bullet}_!(S_M(K_M),L^{M}_{\mu+\rho+w^{-1}(\rho-2\rho_M)})[l(w)-\dim N] \]
only if
\begin{equation} \label{eq:converge1}
\ang{\alpha^{\vee},w(1+\theta)w_0(\mu+\rho))} < 0 \quad \forall \alpha^{\vee} \in \Phi_G^{\vee +} \setminus \Phi_M^{\vee +} \,.
\end{equation}
Here $w_0$ is the longest element of $W_G$.
If the above equation holds and no Eisenstein series arising from $M$ has a
pole at $w_{0,M}w(\mu+\rho)$, then the image contains the cuspidal part.
In particular, the image contains the cuspidal part if \eqref{eq:converge1}
is satisfied and
\begin{equation} \label{eq:converge2}
\left|\ang{\alpha^{\vee},(1+\theta)w_0\mu}\right| \ge 4
|\ang{\alpha^{\vee},\rho}|
\quad \forall \alpha^{\vee} \in \Phi_G^{\vee} \setminus \Phi_M^{\vee} \,.
\end{equation}

The constraint \eqref{eq:converge1} is archimedean in nature, and therefore
appears to provide an obstacle to interpolating Eisenstein series
$p$-adically.  To get around this issue, we will combine contributions from
parabolic subgroups having common Levis.

Let $M$ be a Levi subgroup of $G_{\Qp}$ containing $T$, and
let $\mu$ be a character of $T$ satisfying \eqref{eq:converge2}.
Choosing a parabolic subgroup of $G_{\Qp}$ containing $M$ is equivalent
to choosing a set of positive coroots for $\Phi_G^{\vee} \backslash \Phi_M^{\vee}$.  Hence exactly one of these parabolic subgroups will satisfy \eqref{eq:converge1}.  Call this parabolic $P_{\mu}$.

At the end of section \ref{sub:bsbdy}, we associated each parabolic subgroup of $G$ with
an element of $W_G$; this association determines a map $w \colon \PP \to W_G$.

\begin{lem} \label{lem:nolamdep}~
\[  l(w(P_{\mu})) = \frac{1}{2} \left| (\Phi_G^- \setminus \Phi_M^-) \cap \theta(\Phi_G^+ \setminus \Phi_M^+) \right| \]
\[ (1-\theta)(\rho+w(P_{\mu})^{-1}(\rho-2\rho_M)) = \sum_{\alpha \in (\Phi_G^+ \setminus \Phi_M^+) \cup \theta(\Phi_G^- \setminus \Phi_M^-) } \alpha \]
In particular, $l(w(P_{\mu}))$ and $(1-\theta)\left(\wtoff\right)$ do not depend on $\mu$.
\end{lem}
\begin{proof}
By definition,
\begin{align*}
l(w(P_{\mu})) & = \left| \left\{ \alpha \in \Phi_G \setminus \Phi_M \middle| \ang{\alpha^{\vee},(1+\theta)w_0\mu}<0, \ang{\alpha^{\vee},\mu} < 0 \right\} \right| \\
& = \left| \left\{ \alpha \in \Phi_G \setminus \Phi_M \middle| \ang{\alpha^{\vee},(1+\theta)w_0\mu}<0, \ang{\alpha^{\vee},w_0 \mu} > 0 \right\} \right|
\,.
\end{align*}
Observe that if $\ang{\alpha^{\vee},w_0\mu}<0<\ang{\alpha^{\vee},\theta w_0\mu}$,
then exactly one of the inequalities
\[ \ang{\alpha^{\vee},w_0\mu}<0<\ang{\alpha^{\vee},(1+\theta)w_0\mu}, \quad
\ang{(-\alpha^{\vee} \theta),w_0\mu}<0<\ang{(-\alpha^{\vee} \theta),(1+\theta)w_0\mu} \]
will be satisfied, and otherwise neither will be satisfied.
So
\[ \left| \left\{ \alpha \in \Phi_G \setminus \Phi_M \middle| \ang{\alpha^{\vee},w_0 \mu}<0<\ang{\alpha^{\vee},\theta w_0 \mu} \right\} \right| =2 l(w(P_{\mu})) \,. \]
This proves the first item.
Similarly,
$\rho + w(P_{\mu})^{-1} (\rho-2\rho_M)$ is the sum over
those $\alpha \in \Phi_G \setminus \Phi_M$
satisfying
\[ \left<\alpha^{\vee},w_0 \mu\right> < 0,\quad \left<\alpha^{\vee},(1+\theta)w_0 \mu \right> < 0 \,, \]
and $-\theta\left(\rho + w(P_{\mu})^{-1} (\rho-2\rho_M)\right)$ is the
sum of the $\alpha$ satisfying
\[ \left<\alpha^{\vee},\theta w_0 \mu\right> > 0,\quad \left<\alpha^{\vee},(1+\theta)w_0 \mu \right> > 0 \,. \]
Observe that
$\alpha$ appears in the overall sum if and only if
$\left< \alpha^{\vee},w_0 \mu\right> < 0$ or $\left< \alpha^{\vee},\theta w_0 \mu\right> > 0$.
This is equivalent to $\left< \alpha^{\vee},\mu\right> > 0$ or $\left< \alpha^{\vee},\theta \mu\right> < 0$.
\end{proof}
We will write $l(M)$ for $l(w(P_{\mu}))$ and $\rho(M,\mu)$ for
$\wtoff$.

Consider the set of Levi subgroups $M_p$ of $G_{\Qp}$
containing $T$ that are conjugate to the base change of a relevant Levi
subgroup of $G$.  For such an $M_p$, fix a relevant Levi
$M$ and $g \in G(\Qp)$ satisfying
$M_p = gM_{\Qp}g^{-1}$.
Each weight $\mu$ satisfying \eqref{eq:converge2} determines a parabolic
subgroup of $G_{\Qp}$ containing $M_p$, and a corresponding parabolic
subgroup $P_{\mu}$ of $G$ containing $M$.
Let $\pi_{\mu}$ denote the projection $P_{\mu} \to M$, and define
\[ \KK_{M,\mu} \colonequals \left\{ \pi_{\mu}(xK^px^{-1} \cap P_{\mu}(\Afp)) \cdot I_{M_{\Qp}} \middle|
x \in P_{\mu}(\Afp) \backslash G(\Afp)/K^p \right\} \,. \]
Let $\mathcal{J}_{\mu}$ denote the set of data $(M_p,M,g,K_M)$ with
$(M_p,M,g)$
as above and $K \in \KK_{M,\mu}$.
(We emphasize that for each $M_p$, we include only a single fixed $M$ and $g$,
not all possible $M$ and $g$.)
For ease of notation, we will just
write the data as $(M,K_M)$ below, leaving $M_p$ and $g$ implicit.

Now we are almost ready to write down an analogue of Lemma \ref{lem:interp}
for cusp forms.
The analysis of the last few sections gives us the following identity.
\begin{lem}~ \label{lem:eisprod1}
For any dominant algebraic weight $\mu \colon T \to F^{\times}$
satisfying \eqref{eq:converge2},
\[
\begin{split}
& \frac{\det(1-Xf|H^{\bullet}(S_G(K),L^G_{\mu}))}{\det(1-Xf|H^{\bullet}_!(S_G(K),L^G_{\mu}))} \\
\equiv & \prod_{(M,K_M) \in \mathcal{J}_{\mu}} \det(1-Xf|H^{\bullet}_!(S_M(K_M),L^M_{\mu+\rho(M,\mu)}))^{(-1)^{\dim N-l(M)}} \pmod {\OO_{F} \db{N(\mu,t)X}}
\end{split}
\]
\end{lem}

In order to interpolate the local systems $p$-adically, we need to replace
$\MM_{\mu}$ and $\rho(M,\mu)$ with something independent of $\mu$.
\begin{prop} \label{prop:eisprod2}
For any dominant algebraic weights $\mu \colon T \to F^{\times}$ and
$\mu_0 \colon T \to F_0^{\times}$ satisfying \eqref{eq:converge2},
\[ \begin{split}
& \frac{\det(1-Xf|H^{\bullet}(S_G(K),L^G_{\mu}))}{\det(1-Xf|H^{\bullet}_!(S_G(K),L^G_{\mu}))} \\
\equiv &
\prod_{(M,K_M) \in \mathcal{J}_{\mu_0}}
\det(1-Xf|H^{\bullet}_!(S_M(K_M),L^M_{\mu+\rho(M,\mu_0)}))^{(-1)^{\dim N-l(M)}} \pmod {\OO_{F} \db{N(\mu,t)X}} \,.
\end{split} \]
\end{prop}
\begin{proof}
We claim that local systems $L^M_{\mu+\rho(M,\mu)}$,
$L^M_{\mu+\rho(M,\mu_0)}$ are
isomorphic.  The isomorphism class of each local system depends only the
restriction of the weight to $M^{\der}$.  The operator $\frac{1-\theta}{2}$
acts as the identity on the character lattice of $M^{\der}$,
so the claim follows from Lemma \ref{lem:nolamdep}.
Furthermore, the isomorphism of local systems induces an $\HH_G$-equivariant isomorphism
on cohomology.  (The isomorphism on cohomology is not $\HH_M$-equivariant---the
actions of $u_t$ differ by a factor of $t^{\rho(M,\mu)-\rho(M,\mu_0)}$.
However, the two homomorphisms $\HH_G \to \HH_M$
also differ by the same factor,
and so the differences cancel each other.)

It remains to explain why can replace $\KK_{M,\mu}$ with
$\KK_{M,\mu_0}$.  Essentially, we need to show that if
$\pi = \pi_{\infty} \otimes \pi_p \otimes \pi^p_f$ is an automorphic
representation of $M$, then
\[ \sum_{K_M \in \KK_{M,\mu}} \tr (\mathbf{1}_{K^p_M}|\pi^p_f) = \tr (\mathbf{1}_{K^p}|\Ind_{P_{\mu}(\Afp)}^{G(\Afp)} \pi^p_f) \]
is independent of $\mu$.  By \cite[2.9--2.10]{BZ77}, for any
place $v$, the composition series of the local factor of
$\Ind_{P_{\mu}(\Af)}^{G(\Af)} \pi^p_f$ at $v$ is independent of $\mu$.
It follows that the trace of $\mathbf{1}_{K^p}$ does not depend on $\mu$.
\end{proof}

\subsection{The complex $C^{\bullet}_{G,K^p,\lambda,\cusp}$}
\label{sec:auto-cuspidal}

Now we fix an algebraic dominant weight $\mu_0$, and let
$\lambda \colon T \to A^{\times}$ be any weight.
We define $C^{\bullet}_{G,K^p,\lambda,\cusp}$ inductively, assuming
that analogous complexes have already been defined for $M \in \MM$.
\[
C^{\bullet}_{G,K^p,\lambda,\cusp} \colonequals C^{\bullet}_{G,K^p,\lambda} \oplus \bigoplus_{(M,K_M) \in \mathcal{J}_{\mu_0}} C^{\bullet}_{M,K_M,\lambda+\rho(M,\mu_0),\cusp} [l(M)-\dim N-1]
\]

\begin{prop} \label{prop:interpcusp}
Let $F$ be a finite extension of $\QQ_p$, let
$\mu \colon T \to F^{\times}$ be an algebraic dominant weight, and let
$f=u_t \otimes f^p \in \HH_G'$.  If $\mu$ is sufficiently general, then
\[ \det(1-Xf|C^{\bullet}_{G,K^p,\mu,\cusp}) \equiv \det(1-Xf|H^{\bullet}_{!}(S_G(K),L^G_{\mu})) \pmod {\OO_F \db{N(\mu,t)X}} \,. \]
\end{prop}
\begin{proof}
By induction, we may assume that the proposition holds for all Levi subgroups of $G$.
\[ \begin{split}
& \det(1-Xf|C^{\bullet}_{G,K^p,\mu,\cusp}) \\
\equiv & \det(1-Xf|H^{\bullet}(S_G(K),L^G_{\mu})) \prod_{M, K_M} \det(1-Xf|H^{\bullet}_!(S_M(K_M),L^M_{\mu+\rho(M,\mu_0)}))^{(-1)^{l(M)-\dim N+1}} \\
\equiv & \det(1-Xf|H^{\bullet}_!(S_G(K),L^G_{\mu}))\pmod {\OO_F \db{N(\mu,t)X}}
\end{split} \]
where we used the induction hypothesis and Lemma \ref{lem:interp}
in the second line and Proposition \ref{prop:eisprod2} in the third line.
We also use the fact that $\rho(M,\mu_0)$ is $M$-dominant, and so
$\OO_F \db{N(\mu+\rho(M,\mu_0),t) X} \subseteq \OO_F \db{N(\mu,t) X}$.
\end{proof}

The analysis of section \ref{sec:auto-fredholm} applies equally well to
$C^{\bullet}_{G,K^p,\lambda,\cusp}$.
For any $f \in \HH_G'$, we may define a characteristic power series
$\det \left( 1-Xf \middle| C^{\bullet}_{G,K^p,\lambda,\cusp} \right)$.
If the Fredholm series
$P_+(X) = \prod_i \det \left(1 - Xf \middle| C^i_{G,K^p,\lambda,\cusp} \right)$ has a factorization
$P_+ = Q_+ S_+$ with $Q_+$ a polynomial with invertible leading coefficient, then this factorization
induces a decomposition
$C^{\bullet}_{G,K^p,\lambda,\cusp} = N^{\bullet} \oplus F^{\bullet}$.
\begin{rmk}
One can use Proposition \ref{prop:utfredholm} to show that
for $f \in \HH_G'$,
$\det \left( 1-Xf \middle| C^{\bullet}_{G,K^p,\lambda,\cusp} \right)$
is a Fredholm series.  We will not need to prove this fact for arbitrary
$A$ and $\lambda$, so we leave the details of the argument as an exercise
for the reader.
\end{rmk}
\begin{rmk}
Recall that we took $K_{\infty}$ to be a maximal compact modulo center
subgroup of $G(\RR)$.  In particular, $K_{\infty}$ is not connected in general.
Similarly, $K_{\infty,M}$ is a maximal compact modulo center subgroup of $M(\RR)$.
One could instead take $K_{\infty}$ to be connected (as in e.g. \cite{Han15}).  Then some care is
needed in defining $K_{\infty,M}$---it should not be connected in general.
\end{rmk}

\section{Theory of determinants} \label{sec:determinant}
\numberwithin{equation}{section}

Urban's eigenvariety construction makes use of pseudocharacters.
Chenevier's theory of determinants \cite{Che14} is equivalent to the theory of
pseudocharacters
when the rings involved are $\QQ$-algebras \cite[Proposition 1.27]{Che14}, but
is better behaved in general.
Since we work with rings in which $p$ is not invertible, we will use
determinants.  (However, it is probably not strictly necessary to use
determinants, as we work with rings that are $p$-torsionfree.  See
Corollary \ref{cor:non zero div} and the proof
of Lemma \ref{lem:gluing}.)

We will recall some basic definitions from \cite{Che14}
and prove a lemma concerning the ratio of two determinants.

\begin{defn}[{\cite[\S 1.1-1.5]{Che14}}]
Let $A$ be commutative ring, and let $R$ be an $A$-module.
An \emph{$A$-valued polynomial law on $R$} is a rule
that assigns to any commutative $A$-algebra $B$ a map of sets
$D_B \colon R \otimes_A B \to B$ that is functorial in the sense that
for any $A$-algebra homomorphism $f \colon B \to B'$,
\[ D_{B'} \circ (\id_R \otimes f) = f \circ D_B \,. \]

Let $d$ be a nonnegative integer.
We say that a polynomial law $D$ is \emph{homogeneous of degree $d$}
if
\[ D_B(br)=b^d D_B(r) \quad \forall B\,, b \in B,\, r \in R \otimes_A B \,. \]

Now assume that $R$ is an $A$-algebra.  We say that a polynomial law
$D$ is \emph{multiplicative} if
\[ D_B(1)=1, \quad D_B(rr')=D_B(r) D_B(r') \quad \forall B\,, r,r' \in R \otimes_A B \,. \]

We say that a polynomial law $D$ is a \emph{determinant of dimension $d$}
if it is homogeneous of degree $d$ and multiplicative.
\end{defn}

\begin{ex} \label{module determinant}
Let $M$ be an $R$-module that is projective of rank $d$ as an $A$-module.
Then the rule that sends $r \in R \otimes_A B$ to
$\det(r | M \otimes_A B)$ is a determinant of dimension $d$.
\end{ex}

\begin{lem}[{\cite[Proposition I.1]{Rob63}}] \label{lem:roby}
Let $A$ be a commutative ring, and let $R$ be an $A$-module.
Let $D$ be an $A$-valued polynomial law on $R$ that is homogeneous of
degree $d$, let $n$ be a positive integer, and let $r_1,\dotsc,r_n \in R$.  Then
$D_{A[X_1,\dotsc,X_n]}(X_1 r_1 + \dotsb + X_n r_n)$
is a homogeneous polynomial of degree $d$ in $X_1,\dotsc,X_n$.
\end{lem}

\begin{lem} \label{lem:detratio}
Let $A$ be a commutative ring, let $R$ be an $A$-algebra, and
let $D^+$, $D^-$ be $A$-valued determinants on $R$ of dimension $d_+$, $d_-$,
respectively, with $d_+ \ge d_-$.  Let $d=d_+-d_-$.
There is at most one determinant $D$ of
dimension $d$ satisfying $D_B^+(r) = D_B^-(r) D_B(r)$ for all
$A$-algebras $B$ and all $r \in R \otimes_A B$.

The following are equivalent:
\begin{enumerate}
\item There exists a determinant $D$ satisfying the above condition.
\item For any commutative $A$-algebra $B$
and $r \in R \otimes_A B$, the quotient
\[ D^+_{B[X]}(1+Xr)/D^-_{B[X]}(1+Xr) \]
exists in $B[X]$ and has degree at most $d$.
\item For any positive integer $n$ and $r_1,\dotsc,r_n \in R$,
the quotient
\[ D^+_{A[X_1,\dotsc,X_n]}(1+X_1 r_1 + \dotsb + X_n r_n)/D^-_{A [X_1,\dotsc,X_n]}(1+X_1 r_1 + \dotsb + X_n r_n) \]
exists in $A[X_1,\dotsc,X_n]$ and has total degree at most
$d$.
\end{enumerate}
\end{lem}
\begin{proof}
Let $B$ be a commutative $A$-algebra, and let $r \in R \otimes_A B$.
If $D_B^-(r)$ is not a zero divisor and the quotient $D_B^+(r)/D_B^-(r)$
exists, then we will denote this quotient by
$F_B(r)$.  Note that $D^-_{B[X]}(1+Xr)$ has constant term $1$ by
functoriality with respect to $X \mapsto 0$, so it is not a zero divisor.
Similarly, $D^-_{A[X_1,\dotsc,X_n]}(1+X_1 r_1+\dotsb+X_n r_n)$ has constant
term $1$ and is not a zero divisor.

First, we check that $D$ is uniquely determined if it exists.
Suppose $D$ is a determinant satisfying the conditions of the
lemma.
Let $B$ be a commutative $A$-algebra, and let $r \in R \otimes_A B$.
We claim that the following quantities are equal.
\begin{itemize}
\item $D_B(r)$
\item the coefficient of $X^d Y^0$ in $D_{B[X,Y]}(Y+Xr)$
\item the coefficient of $X^d$ in $D_{B[X]}(1+Xr)$
\end{itemize}
To see that the first and second quantities are equal,
apply functoriality with respect to $X \mapsto 1, Y \mapsto 0$,
using Lemma \ref{lem:roby} to show that $D_{B[X,Y]}(Y+Xr)$ has no $X^n Y^0$
term for $n \ne d$.
To see that the second and third quantities are equal, apply
functoriality with respect to $Y \mapsto 1$, using Lemma \ref{lem:roby}
to show that $D_{B[X,Y]}$ has no $X^d Y^n$ term for $n \ne 0$.
Finally, observe that $F_{B[X]}(1+Xr)$ must exist and must equal
$D_{B[X]}(1+Xr)$.  So $D_B(r)$ must be equal to the coefficient of $X^d$
in $F_{B[X]}(1+Xr)$.  Hence $D$ is uniquely determined if it exists.

Lemma \ref{lem:roby} shows that $(1) \Rightarrow (3)$.

Now we prove $(3) \Rightarrow (2)$.  Assume $(3)$ holds.
Choose a commutative $A$-algebra $B$ and $r \in B \otimes_A R$.
Condition $(3)$ implies that
$F_{A[X_1,\dotsc,X_n]}(1+X_1 r_1 + \dotsb + X_n r_n)$ exists and has
total degree at most $d$.  Then by functoriality with respect to
$X_i \mapsto X b_i$, $F_{B[X]}(1+Xr)$ exists and has degree at most $d$.
This proves $(3) \Rightarrow (2)$.

Now we will show that $(2) \Rightarrow (1)$.  Assume that condition (2) holds.
Define $D_B(r)$ be the coefficient of $X^{d}$ in
$F_{B[X]}(1+Xr)$.  We know that $D^+_{B[X]}(1+Xr)$
(resp. $D^-_{B [X]}(1+Xr)$, $F_{B[X]}(1+Xr)$) has degree
at most $d_+$ (resp. $d_-$, $d$), and we have already showed that the
coefficient of $X^{d_+}$ (resp. $X^{d_-}$, $X^{d}$) is $D_B^+(r)$
(resp. $D_B^-(r)$, $D_B(r)$).  So $D_B^+(r)=D_B^-(r) D_B(r)$.

It remains to show that $D$ is a determinant.
Since $D^+$ and $D^-$ are functorial, $D$ is as well.
To show that $D$ is homogeneous of degree $d$, observe that
the map $X \mapsto bX$ multiplies the coefficient of $X^d$
in $F_{B[X]}(1+Xr)$ by $b^d$.

Finally, we check that $D$ is multiplicative.
We have $F_{B[X]}(1+X)=(1+X)^{d}$, so $D_B(1)=1$.
Observe that $D_B(r_1) D_B(r_2)$ is the coefficient of $(X_1 X_2)^d$ in
$F_{B[X_1,X_2]}(1+X_1 r_1 + X_2 r_2 + X_1 X_2 r_1 r_2)$.
This is the same as the coefficient of $X_1^0 X_2^0 X_3^d$ in
$F_{B[X_1,X_2,X_3]}(1+X_1 r_1 + X_2 r_2 + X_3 r_1 r_2)$, since $X_3^d$
is the only monomial of total degree $d$ in $B[X_1,X_2,X_3]$ that maps
to $(X_1 X_2)^d$ under $X_3 \mapsto X_1 X_2$.
Then applying $X_1 \mapsto 0$, $X_2 \mapsto 0$, we find that
$D_B(r_1) D_B(r_2)$ is the coefficient of $X_3^d$ in
$F_{B[X_3]}(1+X_3 r_1 r_2)$, which is $D_B(r_1 r_2)$.
This concludes the proof that $(2) \Rightarrow (1)$.
\end{proof}

\begin{cor} \label{det injective}
Retain the notation of Lemma \ref{lem:detratio}.
Let $A \hookrightarrow A'$ be an injective map of commutative rings.  Suppose that
there exists an $A'$-valued determinant $D'$ on $R \otimes_A A'$ satisfying
$D_B^+(r) = D_B^-(r) D'_B(r)$ for any $A'$-algebra $B$ and $r \in R \otimes_A B$.
Then there exists an $A$-valued determinant $D$ on $R$ satisfying
$D_B^+(r) = D_B^-(r) D_B(r)$ for any $A$-algebra $B$ and $r \in R \otimes_A B$.
\end{cor}
\begin{proof}
Apply the equivalence $(1) \Leftrightarrow (3)$ of Lemma \ref{lem:detratio}.
Observe that
$F_{A[X_1,\dotsc,X_n]}(1+X_1 r_1+\dotsb+X_n r_n)$ exists and has degree total degree $\le d$ iff
$F_{A\db{X_1,\dotsc,X_n}}(1+X_1 r_1 + \dotsb+X_n r_n)$ is a polynomial of total degree $\le d$.
Since $A \hookrightarrow A'$ is injective, if $F_{A'\db{X_1,\dotsc,X_n}}(1+X_1 r_1 + \dotsb+X_n r_n)$
is a polynomial of total degree $\le d$, then $F_{A\db{X_1,\dotsc,X_n}}(1+X_1 r_1 + \dotsb+X_n r_n)$
is as well.
\end{proof}

\begin{defn}[{\cite[\S 1.17, Lemma 1.19(i)]{Che14}}]
Let $D$ be an $A$-valued determinant on $R$.  We denote by $\ker(D)$
the set of $r \in R$ such that for all $B$ and all
$r' \in B \otimes_A R$, $D_B(1+r'r)=1$.
\end{defn}
\begin{rmk}
Let $M$ be projective $A$-module of rank $d$,
let
$\rho \colon R \to \End(M)$ be a homomorphism,
and let $D$ be the determinant
associated with $\rho$, as in Example \ref{module determinant}.
Then $\ker \rho \subseteq \ker D$.
Conversely, if $r \in \ker D$, then $D_{A[X]}(X-r)=X^d$,
so
$r^d \in \ker \rho$ by the Cayley-Hamilton theorem.
\end{rmk}
\begin{rmk}
Chenevier also defines the Cayley-Hamilton ideal $\operatorname{CH}(D)$.
Assume $D$ comes from a homomorphism $\rho \colon R \to \End(M)$.
Then $\operatorname{CH}(D) \subseteq \ker \rho$.
So we might think of $\ker D$ as an upper bound for $\ker \rho$
and $\operatorname{CH}(D)$ as a lower bound.
If $A$ is Noetherian, then since $\ker \rho \subseteq \ker D$, $R/\ker D$
is a finite $A$-module.  However, $R/\operatorname{CH}(D)$ need not be a
finite $A$-module, making it more difficult to use
$\operatorname{CH}(D)$ in the construction of eigenvarieties.

For a concrete example, consider $A=\QQ$, $M=\QQ^2$,
$R=\QQ[T_1,T_2,\dotsc]$, and let $\rho$ be a map that sends
each $T_i$ to a nilpotent upper triangular matrix.
Then $\ker D = (T_1,T_2,\dotsc)$ (so $R/\ker D \cong \QQ$),
$\ker \operatorname{CH}(D) = (T_1,T_2,\dotsc)^2$
(so $R/\operatorname{CH}(D)$ is not finite type over $\QQ$),
and $R/\ker \rho$ is isomorphic to either $\QQ$ or
$\QQ[\epsilon]/(\epsilon^2)$.
\end{rmk}
\numberwithin{equation}{subsection}

\section{Construction of the eigenvariety} \label{sec:eigenvariety}

\subsection{Weight space and Fredholm series}

Now we are ready to define the eigenvariety following \cite[\S 5]{Urb11}.
We will use Huber's theory of adic spaces \cite{Hub93, Hub94, Hub96};
see also \cite[\S 2-5]{berkeley} for a modern introduction.
Some aspects of our approach follow \cite[Appendice B]{AIP15}.

We return to the setup of sections
\ref{sec:automorphic}--\ref{sec:eisenstein}.
We continue to assume that $G(\RR)$ has discrete series.
Let $T'$ be the quotient of $T_0$ by the closure of $Z_G(\QQ) G_{\infty}^+ K^p \cap T_0$.
We define the weight space
\[ \WW \colonequals \Spa(\Zp \db{T'},\Zp \db{T'})^{\an} \,. \]

Let $\UU = \Spa(A,A^+)$ be an open affinoid subset of $\WW$ with $A$
a complete Tate $\Zp$-algebra (which is automatically Noetherian).
Let $\lambda \colon T_0 \to A^{\times}$ be the tautological character induced by the
map $T_0 \to T' \to \Zp \db{T'}$.

For any $f \in \HH'_G \otimes_{\Zp} A$, let
\[ P_f(X) \colonequals \det \left( 1-Xf \middle| C^{\bullet}_{G,K^p,\lambda,\cusp} \right)^{(-1)^{d/2}} \,. \]
Note that $d=\dim S_G(K)$ is even since $G(\RR)$ has discrete series.
If $\VV$ is an open subspace of $\WW$, and $f \in \HH'_G \otimes_{\Zp} \OO_{\WW}(\VV)$,
then we define $P_f(X)$ by gluing.

\begin{defn}
Let $\VV$ be an open subspace of $\WW$.  A series
$f \in \OO_{\WW}(\VV) \db{X}$
is called a Fredholm series if it is the power series expansion of some
global section of $\VV \times \mathbb{A}^1$ and its leading coefficient is $1$.
\end{defn}
This definition agrees with Definition \ref{def:fredholmtate} if
$\VV=\Spa(A,A^+)$ with $A$ a complete Tate $\Zp$-algebra.

\begin{prop} \label{prop:utfredholm}
For $f \in \HH_G'$,
the series $P_{f}(X) \in \OO_{\WW}(\WW)\db{X}$ is a Fredholm series.
\end{prop}
\begin{proof}
Observe that $\OO_{\WW}(\WW) = \OO_{\WW}^+(\WW) = \Zp \db{T'}$.
Let $T'_{\mathrm{tf}}$ be a maximal torsionfree subgroup of $T'$.
The topology on $\Zp \db{T'_{\mathrm{tf}}}$ is induced by any norm corresponding
to a Gauss point of the wide open polydisc
$\Spa(\Zp \db{T'_{\mathrm{tf}}}, \Zp \db{T'_{\mathrm{tf}}}) \times_{\Spa(\Zp,\Zp)} \Spa(\Qp,\Zp)$.
Similarly, the topology on $\Zp \db{T'}$ is induced by
a supremum of a finite collection of norms corresponding to Gauss points of
$\WW \times_{\Spa(\Zp,\Zp)} \Spa(\Qp,\Zp)$.
So it suffices to check that the restrictions of $P_{f}(X)$
to Gauss points of $\WW$ are Fredholm series.
The Gauss points are characteristic zero points, so we may apply
the argument of \cite[Theorem 4.7.3iii]{Urb11} along with Proposition
\ref{prop:interpcusp}.
\end{proof}

We will write $P(X)$ for $P_{u_t}(X)$.
We define the spectral variety $\ZC \subset \WW \times \AD^1$ to be the zero
locus of $P(X)$,
and we define $w \colon \ZC \to \WW$ to be the projection.
We also define
\[ P_+(X) \colonequals \prod_i \det \left( 1-Xu_t \middle| C^{i}_{G,K^p,\lambda,\cusp} \right) \,. \]

\subsection{Weight space and its its characteristic zero subspace}

Before constructing the eigenvariety, we will prove a result that will
allow us to deduce information about the behavior of the eigenvariety
at the boundary from the characteristic zero part of the eigenvariety.
\begin{lem} \label{lem:non zero div}
Let $\UU = \Spa(A,A^+)$ by an affinoid adic space.  Assume that
$A$ is finitely generated over a Noetherian ring of definition.
Let $a \in A$ be an element that is not a zero divisor.
\begin{enumerate}
\item For any open $\VV \subseteq \UU$, $a$ is not a zero divisor in
$\OO_{\UU}(\VV)$.
\item Assume $A$ is Tate.  There exists rational subset $\VV \subseteq \UU$
such that the restriction $A \to \OO_{\UU}(\VV)$ is injective
and $a \in \OO_{\UU}(\VV)^{\times}$.
\end{enumerate}
\end{lem}
\begin{proof}
To prove the first item, it suffices to consider the case where
$\VV$ is a rational subset.  Then $\OO_{\UU}(\UU)$ is flat over $A$
by \cite[Corollary 1.7(i)]{Hub93}, and $\OO_{\UU}(\VV)$ is flat over
$\OO_{\UU}(\UU)$ by \cite[Proposition 1.6.7(i), Lemma 1.7.6]{Hub96}.
Since the multiplication-by-$a$ map is injective on $A$, it must
be injective on $\OO_{\UU}(\VV)$ as well.

To prove the second item, choose a Noetherian ring of definition $A_0 \subset A$
and a topologically nilpotent $\alpha \in A^{\times} \cap A_0$.
After multiplying $a$ by a power of $\alpha$, we may assume $a \in A_0$.
By the Artin-Rees lemma, there exists an integer $k \ge 1$ so that
$\alpha^n A_0 \cap a A_0 \subseteq \alpha^{n-k} aA_0$
for all $n \ge k$.  Let $\VV \subset \UU$ be the rational subset defined
by the inequality $|a| \ge |\alpha^k|$.
We claim that $A \to \OO_{\UU}(\VV)$ is injective.

The ring $\OO_{\UU}(\VV)$ is the $\alpha$-adic completion
of $A_0[1/a]$, and the completion of $A_0[\alpha^k/a]$ is a ring of
definition.
Let $b \in A$, and suppose the image of $b$ in $\OO_{\UU}(\VV)$ is zero.
Then for each $n \in \NN$, $b \in \alpha^n A_0[\alpha^k/a]$.
Then there exists $m \in \NN$ so that
$a^m b \in \alpha^n (a,\alpha^k)^{m} A_0$.
One can then show by induction on $m$ that
$b \in \alpha^n A_0$.
Since $A_0$ is $\alpha$-adically separated, this
implies $b=0$.
\end{proof}
\begin{cor} \label{cor:non zero div}
Let $\UU = \Spa(A,A^+)$ be an open affinoid subspace of $\WW$,
with $A$ complete Tate.
Then $p$ is not a zero divisor of $\OO_{\WW}(\UU)$, and there
exists a rational subset $\VV \subseteq \UU$ such that
$A \to \OO_{\WW}(\VV)$ is injective and $p \in \OO_{\WW}(\VV)^{\times}$.
\end{cor}

\subsection{Pieces of the eigenvariety}

Now we construct the individual pieces of the eigenvariety.
Let $z \in \ZC$.  By \cite[Corollaire B.1]{AIP15}, there exists
an open affinoid neighborhood $\UU = \Spa(A,A^+)$ of $w(z)$
and a factorization $P_+(X)=Q_+(X) S_+(X)$, with
$Q_+(X) \in A[X]$, $S_+(X) \in A \fs{X}$,
such that $Q_+(X)$ and $S_+(X)$ are relatively prime,
$Q_+$ vanishes at $x$, and the leading coefficient of $Q_+$ is invertible.
The factorization of $P_+$ induces a factorization
$P(X)=Q(X) S(X)$ satisfying similar properties.
The factorization also determines a sub-complex $N^{\bullet}$
of $C^{\bullet}_{G,K^p,\lambda,\cusp}$, as described in sections
\ref{sec:auto-fredholm} and \ref{sec:auto-cuspidal}.

\begin{prop} \label{prop:detexists}
Let $D^+$ be the determinant associated with the action of
$\HH_G \otimes_{\Zp} A$ on $\bigoplus_{i \equiv d/2 (2)} N^i$,
and let $D^-$ be the determinant associated with the action of
$\HH_G \otimes_{\Zp} A$ on $\bigoplus_{i \equiv d/2+1 (2)} N^i$.
Then there exists a determinant $D$ so that $D^+=D^- D$.
\end{prop}
\begin{proof}
Let $R = \HH_G \otimes_{\Zp} A$.
As in Lemma \ref{lem:detratio}, if $B$ is an $A$-algebra and $r \in R \otimes_A B$
such that $D_B^-(r)$ is not a zero divisor and the ratio $D_B^+(r)/D_B^-(r)$ exists in $B$, we write
$F_B(r)$ for this ratio.
Let $d_+$ and $d_-$ be the dimensions of $D^+$ and $D^-$, respectively,
and let $d=d_+-d_-$.

By Corollary \ref{cor:non zero div},
we can find a rational subset $\VV \subset \UU$ so that the restriction
$A \to \OO_{\WW}(\VV)$ is injective and $p$ is invertible on $\VV$.
Since $\VV$ is a reduced rigid space, the natural map
$\OO_{\WW}(\VV) \to \prod_{x} k_x$ is injective, where the product runs
over rigid analytic points $x \in \VV$ and $k_x$ is the residue field of $x$.
For each $x$,
write $D^+|_{k_x}$ (resp.\ $D^-|_{k_x}$) for the base change of
$D^+$ (resp.\ $D^-$) along $A \to k_x$.
By \cite[Lemma 4.1.12 and Theorem 4.7.3iii]{Urb11},
the difference of the pseudocharacters corresponding to
$D^+|_{k_x}$ and $D^-|_{k_x}$ is again a pseudocharacter.
By the equivalence of
pseudocharacters and determinants in characteristic zero
\cite[Proposition 1.27]{Che14}, there exists a determinant $D|_{k_x}$
satisfying $D^+|_{k_x} = D^-|_{k_x} D|_{k_x}$.
Then by Corollary \ref{det injective}, there exists
a determinant $D$ satisfying $D^+ = D^- D$.
\end{proof}

Let
\[ h_{\UU,Q_+} \colonequals (\HH_G \otimes_{\Zp} A)/\ker(D) \,. \]
We will use the extension $A \to h_{\UU,Q_+}$ to construct
an adic space $\EE_{\UU,Q_+}$ over $\UU$.
\begin{lem}
The ring $h_{\UU,Q_+}$ is a finite $A$-module.
\end{lem}
\begin{proof}
Since $\ker(D)$ contains any operator that
annihilates $N^{\bullet}$, $h_{\UU,Q_+}$ can be identified with a subquotient of
$\bigoplus_i \End N^i$.  In particular, $h_{\UU,Q_+}$ must be finitely generated as an
$A$-module.
\end{proof}
We give $h_{\UU,Q_+}$ the ``$A$-module topology'' defined in
\cite[Section 2]{Hub94}.
\begin{lem}
The ring $\hu$ is Tate and has a Noetherian ring of definition.
\end{lem}
\begin{proof}
Choose $A$-module generators $a_1,\dotsc,a_n$ of $\hu$.  Choose $m_{ijk} \in \hu$
so that for each $i,j$, $a_i a_j = \sum_{k=1}^n m_{ijk} a_k$.
Let $A_0$ be a ring of definition of $A$, and let $\alpha$ be a topologically nilpotent unit of $A$
contained in $A_0$.
There exists an integer $\ell$ so that $\alpha^{\ell} m_{ijk} \in A_0$ for all $i,j,k$.
Let $\huo$ be the $A_0$-submodule of $\hu$ generated by $1,\alpha^{\ell} a_1,\dotsc,\alpha^{\ell} a_n$;
then $\huo$ is an open subring of $\hu$.  Then $\huo$ is Noetherian since $A_0$ is Noetherian, and $\huo$
inherits the $\alpha$-adic topology from $A_0$.  So $\hu$ has a Noetherian ring of definition and is Tate.
\end{proof}
Let $\hu^+$ be the normal closure of $A^+$ in $\hu$.
Then $(\hu,\hu^+)$ is a Huber pair.
We define $\EE_{\UU,Q_+} \colonequals \Spa(h_{\UU,Q_+},h_{\UU,Q_+}^+)$.

Since $Q^*(X)$ is the characteristic polynomial of $u$ acting on $N^{\bullet}$,
it follows from \cite[Lemma 1.12iv]{Che14} that $Q^*(u)$ is in
$\ker(D)$, and so there is a canonical map $\EE_{\UU,Q_+} \to \ZC$.

\subsection{Gluing}

We will glue the $\EE_{\UU,Q_+}$ as in \cite[Section 5]{Buz07}.  We need the
following lemma to verify that the pieces can be glued.

\begin{lem} \label{lem:gluing}
If $\UU' \subset \UU$ are affinoid subspaces of $\WW$, then there is a canonical isomorphism
$\EE_{\UU',Q_+} \cong \EE_{\UU,Q_+} \times_{\UU} \UU'$.
\end{lem}
\begin{proof}
By Corollary \ref{cor:non zero div}, $p$ is not a zero divisor in
$\OO_{\WW}(\UU)$ and $\OO_{\WW}(\UU')$.
Hence $\OO_{\WW}(\UU)$ and $\OO_{\WW}(\UU')$ are torsionfree $\ZZ$-modules.
The map $\OO_{\WW}(\UU) \to \OO_{\WW}(\UU')$ is flat by
\cite[Lemma 1.7.6]{Hub96}.  By the argument of \cite[Proposition I.2.2.4]{Ryd08},
if $A$ is a ring that is a torsionfree $\ZZ$-module, then the
kernel of an $A$-valued determinant is the same as the kernel of
the associated pseudocharacter.
By \cite[Proposition I.2.2.8]{Ryd08}, the formation of the kernel of
a pseudocharacter commutes with flat base change.
So the formation of the kernel of a determinant commutes with the base change
$\OO_{\WW}(\UU) \to \OO_{\WW}(\UU')$.
Then
$h_{\UU',Q_+} \cong h_{\UU,Q_+} \otimes_{\OO_{\WW}(\UU)} \OO_{\WW}(\UU')$.
Since $h_{\UU,Q_+}$ is finite over $\OO_{\WW}(\UU')$, it follows that
$\EE_{\UU',Q_+} \cong \EE_{\UU,Q_+} \times_{\UU} \UU'$.
\end{proof}

One can also show, using essentially the same proof as
\cite[Proposition 5.3.5]{Urb11}, that if $Q_+$ and $Q_+'$ are relatively
prime, then there is a canonical isomorphism
$\EE_{\UU,Q_+ Q_+'} \cong \EE_{\UU,Q_+} \sqcup \EE_{\UU,Q_+'}$.

\begin{thm} \label{thm:eigenvariety}
The $\EE_{\UU,Q_+}$ can be glued to form an adic space $\EE$.
Furthermore, $\EE$ is equidimensional in the sense of
\cite[Definition 1.8.1]{Hub96} and the morphism $\EE \to \ZC$ is finite and surjective.
\end{thm}
\begin{proof}
To show that the morphism $\EE \to \ZC$ is finite, we observe that
$\ZC$ can be covered by open sets whose preimage in $\EE$ is contained
in some $\EE_{\UU,Q_+}$.  The finiteness of the morphism $\EE \to \ZC$ then
follows from the finiteness of the maps $\EE_{\UU,Q_+} \to \UU$.

Now we check that the morphism is surjective.
Let $z \in \ZC$, and let $k$ be the residue field of $z$.
Observe that $\operatorname{Spec} \HH_G \to \operatorname{Spec} \Zp[u_t]$ is
surjective, so $\HH_G \otimes_{\Zp[u_t]} k$ cannot be the zero ring.
The image of $\ker(D)$ in $\HH_G \otimes_{\Zp[u_t]} k$ is contained
in the kernel of the base change of $D$ to $\HH_G \otimes_{\Zp[u_t]} k$.
Therefore the image of $\ker(D)$ cannot be the unit ideal, and so there
must be a point of $\EE$ lying above $z$.

Finally, we show that $\EE$ is equidimensional.
The weight space $\WW$ has the same dimension as its characteristic zero
part.
By \cite[Theorem 5.3.7(iii)]{Urb11}, the characteristic zero part of
$\EE$ is equidimensional of dimension $\dim \WW$.
By Lemma \ref{lem:non zero div}(2), any nonempty open $\UU \subseteq \EE$
has nonempty characteristic zero part, so it has dimension at least $\dim \WW$.
Conversely, since $\EE$ is locally finite over $\WW$, $\EE$ has
dimension at most $\dim \WW$ by \cite[Example 1.8.9(ii)]{Hub96}.
\end{proof}

\bibliography{../refs}{}
\bibliographystyle{../shortalpha}
\end{document}